\documentclass[pdftex,12pt]{article}
\usepackage{amsmath,amsfonts,amssymb,epsfig,epic,graphics,pictex}
\usepackage{graphicx}

\begin{document}
\newtheorem{lem}{Lemma}[section]
\newtheorem{prop}{Proposition}
\newtheorem{con}{Construction}[section]
\newtheorem{defi}{Definition}[section]
\newtheorem{coro}{Corollary}[section]
\newcommand{\hf}{\hat{f}}
\newtheorem{fact}{Fact}[section]
\newtheorem{theo}{Theorem}
\newcommand{\Br}{\Poin}
\newcommand{\Cr}{{\bf Cr}}
\newcommand{\dist}{{\rm dist}\,}
\newcommand{\diam}{{\rm{diam}}\, }
\newcommand{\ord}{{\rm ord}\, }
\newcommand{\Dm}{{\rm Dm}\,}
\newcommand{\compose}{\circ}
\newcommand{\dbar}{\bar{\partial}}
\newcommand{\Def}[1]{{{\em #1}}}
\newcommand{\dx}[1]{\frac{\partial #1}{\partial x}}
\newcommand{\dy}[1]{\frac{\partial #1}{\partial y}}
\newcommand{\Res}[2]{{#1}\raisebox{-.4ex}{$\left|\,_{#2}\right.$}}
\newcommand{\sgn}{{\rm sgn}}
\def\br#1{\left(#1\right)}
\def\Bbb#1{\mbox{\mathfontb #1}}
\newcommand{\CC}{\mathbb{C}}
\newcommand{\DD}{{\mathbb D}}
\newcommand{\ovB}{{\overline{B}}}
\newcommand{\RR}{\mathbb{R}}
\newcommand{\NN}{\mathbb{N}}
\newcommand{\Z}{\mathbb{Z}}
\newcommand{\Q}{\mathbb{Q}}
\newcommand{\QQ}{\mathbb{Q}}
\newcommand{\ZZ}{\mathbb{Z}}
\newcommand{\tr}{\mbox{Tr}\,}
\newcommand{\C}{\mathbb{C}}
\def\grtsim{\mbox{\mathfonta\&}}
\def\gtrsim{\mbox{\mathfonta\&}}
\def\lesssim{\mbox{\mathfonta.}}
\def\D{\Bbb D}
\def\T{\mathcal{T}}
\def\K{\cal K}
\def\M{{\cal M}_d}
\newcommand{\area}{{\rm{area}}\,}
\newcommand{\HD}{\mbox{HD}\,}
\newcommand{\J}{\mathcal{J}}

\newenvironment{nproof}[1]{\trivlist\item[\hskip \labelsep{\bf Proof{#1}.}]}
{\begin{flushright} $\square$\end{flushright}\endtrivlist}
\newenvironment{proof}{\begin{nproof}{}}{\end{nproof}}

\newenvironment{block}[1]{\trivlist\item[\hskip \labelsep{{#1}.}]}{\endtrivlist}
\newenvironment{definition}{\begin{block}{\bf Definition}}{\end{block}}

\newtheorem{conjec}{Conjecture}

\newtheorem{com}{Comment}
\font\mathfonta=msam10 at 11pt
\font\mathfontb=msbm10 at 11pt
\def\Bbb#1{\mbox{\mathfontb #1}}
\def\lesssim{\mbox{\mathfonta.}}
\def\suppset{\mbox{\mathfonta{c}}}
\def\subbset{\mbox{\mathfonta{b}}}
\def\grtsim{\mbox{\mathfonta\&}}
\def\gtrsim{\mbox{\mathfonta\&}}

\newcommand{\Poin}{{\bf Poin}}
\newcommand{\Bo}{\Box^{n}_{i}}
\newcommand{\Di}{{\cal D}}
\newcommand{\gd}{{\underline \gamma}}
\newcommand{\gu}{{\underline g }}
\newcommand{\ce}{\mbox{III}}
\newcommand{\be}{\mbox{II}}
\newcommand{\F}{\cal{F}}
\newcommand{\Ci}{\bf{C}}
\newcommand{\ai}{\mbox{I}}
\newcommand{\dupap}{\partial^{+}}
\newcommand{\dm}{\partial^{-}}
\newenvironment{note}{\begin{sc}{\bf Note}}{\end{sc}}
\newenvironment{notes}{\begin{sc}{\bf Notes}\ \par\begin{enumerate}}%
{\end{enumerate}\end{sc}}
\newenvironment{sol}
{{\bf Solution:}\newline}{\begin{flushright}
{\bf QED}\end{flushright}}

\title{Analytic structures and harmonic measure at bifurcation locus}

\author{Jacek Graczyk\\
\small{Department of Mathematics}\\
\small{University of Paris XI}\\
\small{91405 Orsay Cedex}\\
\small{France}\\
\and
Grzegorz \'{S}wia\c\negthinspace tek
\thanks{Supported in part by Narodowe Centrum Nauki - grant 2015/17/B/ST1/00091.}\\
\small{Department of Mathematics and Information Science}\\ 
\small{Warsaw Technical University}\\
\small{ul. Koszykowa 75, 00-662 Warsaw}\\
\small{Poland}}
\maketitle
\abstract{We study conformal quantities at generic parameters with respect to the harmonic measure on the boundary of 
the connectedness loci ${\cal M}_d$ for unicritical polynomials $f_c(z)=z^d+c$. It is known that these parameters are structurally unstable and 
have stochastic dynamics. We prove $C^{1+\frac{\alpha}{d}-\epsilon}$-conformality, $\alpha = 2-\HD({\cal J}_{c_0})$,
of the parameter-phase space similarity maps $\Upsilon_{c_0}(z):\C\mapsto \C$  at typical $c_0\in \partial {\cal M}_d$
and establish that globally quasiconformal similarity maps $\Upsilon_{c_0}(z)$, $c_0\in \partial {\cal M}_d$, are 
$C^1$-conformal along external rays landing at
$c_0$ in $\C\setminus {\cal J}_{c_0}$ mapping onto the corresponding rays of ${\cal M}_d$. This conformal equivalence leads to the proof that 
the $z$-derivative of the similarity map $\Upsilon_{c_0}(z)$ 
at typical $c_0\in \partial {\cal M}_d$ is equal  
to $1/{\cal T}(c_0)$, where ${\cal T}(c_0)=\sum_{n=0}^{\infty}(D(f_{c_0}^n)(c_0))^{-1}$ is the transversality function. 

The paper builds analytical tools for a further study of  the extremal properties  of the harmonic measure on $\partial \M$, \cite{grasw}. In particular,
we will explain how a  non-linear dynamics creates abundance of hedgehog neighborhoods in $\partial \M$ effectively blocking a good  access of $\partial \M$ from the outside.}

\section{Introduction}
One of the main open problems in dynamical systems is the density of hyperbolic polynomials in the complex plane~\cite{fat, yoc1, smale}. 
Even in the simplest case of quadratic  polynomials $z^2+c$, the problem is far from being solved inspite of great deal of 
research. The main object of this study is the boundary of connectedness locus ${\cal M}_d$, $d\geq 2$, and its 
relations with the corresponding Julia sets ${\cal J}_c$ through local similarity maps. 
Understanding the fractal structure of ${\cal M}_d$  which is both  "self-similar" and "chaotic" is one of 
the most interesting  aspects of complex dynamics.

Since ${\cal M}_d$ is a {\em full} compact~\cite{etiuda,sibony}, Carath\'eodory's theorem implies that local connectivity of $\partial \M$ is equivalent  to the existence  of continuous extension of 
the Riemman map $\Psi:\hat\CC\setminus \D\mapsto \hat\CC\setminus {\cal M}_d$ tangent to the identity at $\infty$. By~\cite{etiuda},
the local connectivity of $\partial {\cal M}_d$ implies the density of hyperbolicity. This is not known, nor it is 
known whether every hyperbolic geodesic in $\CC\setminus {\cal M}_d$ lands. 

The Julia set ${\cal J}_{c}$ of a unicritical polynomial $f_{c}(z)=z^{d}+c$
is defined as the closure of all repelling periodic points of $f_{c}$, 
$$\J_c = \overline{\{z\in \C: \exists n\in \NN\;\; f_{c}^{n}(z)=z \mbox{ and }
|(f_{c}^{n})'(z)|>1\}}\;. $$

Let ${\cal M}_{d}$ be the set of all $c\in \C$ 
for which ${\cal J}_{c}$ is connected. When $c$ is outside $\M$ then Julia sets ${\cal J}_c$ are   totally disconnected. 
The boundary of ${\cal M}_{d}$ is {\em the topological bifurcation locus} of ${\cal J}_{c}$.

A mathematical interest in $\M$ goes beyond unicritical dynamics. C. McMullen  proved in~\cite{macun}  
that the bifurcation locus of any non-trivial holomorphic family of rational maps over the unit disk contains almost conformal copies of $\M$.
The distribution of the harmonic measure on
$\partial \M$ has  some extremal properties
and  it is in the same time  computationally accessible. For various relations with classical problems in complex analysis see~\cite{bs, bms, jo, puz, fine}.

Our goal is to  develop analytical tools to understand how dynamics unfolds
at generic parameters with respect to the harmonic measure on $\partial \M$. The current work is based on two ingredients which fit well into a program of J.-C. 
Yoccoz to study  the parameter space $\M$ through interaction of analytic and combinatorial structures. The first ingredient is an organization 
of the parameter space into a Markov system, proposed by J.-C. Yoccoz, which is asymptotically  stable with respect to the natural  holomorhic motions. 
The second one stems from~\cite{fine} and combines an outside combinatorics given by holomorphically moving B\"otker coordinates with some simple probabilistic models
of~\cite{grsw}. This ``outside'' approach of ~\cite{grsw, sm} allows to control the behavior of the harmonic measure on $\partial \M$ and use the 
results from harmonic analysis and 
potential theory to get a further information about the dynamics ~\cite{gniotek}. Another pertinent examples include~\cite{man,zak}. 

In~\cite{fine}, a system of  similarity functions $\Upsilon_{c_0}(z)$ was constructed 
which is parame\-trized by typical points $c_0\in \partial \M$ with respect to the harmonic measure. The  maps $\Upsilon_{c_0}(z)$ are quasiconformal 
and become asymptotically conformal. The main result of the paper, stated as Theorem~\ref{theo:continuity}, asserts that the similarity maps 
$\Upsilon_{c_0}(z)$ depend $C^1$-continuously along hyperbolic geodesics landing at $c_0$ even if the geodesics have a rather complicated geometry, 
spiralling in both directions infinitely many times~\cite{fine}. However, the oscillations can be controlled asymptotically by a universal function of 
Theorem~\ref{theo:flat}. The idea of comparing the phase  and parameter spaces dates back to the origins of complex  dynamics~\cite{etiuda}
and the similarity theorem  of T. Lei~\cite{tan} was one of  the first results in the area inspired by computer visualisations. 

Theorem~\ref{theo:continuity} has several applications but the most interesting  direction  from the point of view of the
dynamical systems is given  by the formula which relates the derivative of $\Upsilon_{c_0}(z)$ with the transversality function ${\cal T}(c)$, cf. formula
(\ref{equ:26mp,1}), introduced by M. Benedicks and L. Carleson 
in~\cite{Beca,Becaa}. In the Misiurewicz case, 
the formula was proven by J. Riviera-Letelier in~\cite{rivera} by methods exploiting  an underlying hyperbolicity. 
The Misiurewicz parameters  are defined by the condition that the critical point is not recurrent and thus of bounded combinatorial complexity.
The parameter selection methodes of Benedicks-Carleson~\cite{Beca, Becaa} rely on the fact that  ${\cal T}(c_0)\not = 0$ at the parameter $c_0$ which undergoes a perturbation.
Since  ${\cal T}(c)$ is analytic for $c$ from the complement of $\M$, it can not take the same value on large sets. The size of these sets will depend on 
integrability properties of  ${\cal T}(c)$~\cite{ahbe, car}. On the other hand, the failure of the transversality condition ${\cal T}(c_0)\not = 0$ endows dynamics 
with some weak expansion properties. If the series ${\cal T}(c)=\sum_{n=0}^{\infty}(Df_c^{n}(c))^{-1}$ converges absolutely then ${\cal J}_c$ is 
locally connected~\cite{rish, grsm1} and of Lebesgue measure zero~\cite{brstr}. 
The perturbative techniques of ~\cite{Beca,Becaa} were already applied in the complex quadratic setting in~\cite{bengra} and proven to work well in various 
holomorphic instances by  M. Aspenberg~\cite{Asp, Asp1}. A direct relation of the transversality function to  the Fatou conjecture for unicritical polynomials 
was shown by G. Levin in~\cite{lev}.  The transversality condition ${\cal T}(c_0)\not = 0$ was also intensively studied for real maps in the context of 
Jakobson's theorem~\cite{jak}, see for example~\cite{tsuji}, where the transversality condition was explicitely stated. 
The real methods are quite different than these adopted in the current paper and will not be further discussed.

\subsection{Similarity structures}
By Fatou's theorem, $\Psi:\hat\CC\setminus \D\mapsto \hat\CC\setminus {\cal M}_d$, tangent to the identity at $\infty$, extends radially almost everywhere on the unit circle with respect to the normalized $1$-dimensional
Lebesgue measure $\lambda_1$.  The harmonic measure $\omega$ on $\partial \M$ is equal to $\Psi_{*}(\lambda_1)$. If $c\in \partial \M$ then $\J_c$ is a full compact. 
Denote by $\Psi_c:\hat\CC\setminus \D\mapsto \hat\CC\setminus {\cal J}_c$ the Riemann map tangent  to the identity at $\infty$. We have a one parameter  family $\omega_{c\in \partial \M}$ of the  harmonic measures supported on the corresponding Julia sets ${\cal J}_c$,
$\omega_c=(\Psi_c)_{*}(\lambda_1)\;.$
 If $c\in \partial \M$ is typical with respect to $\omega$ then the same is true for the critical orbit  $\{f_c^n(c)\}_{n\in \NN}$  with respect to the harmonic measure on ${\cal J}_c$, $\omega_c$ is also $f_c$-invariant, ergodic, and of the maximal entropy $\log d$~\cite{bro}. 
 
The multifaceted relation  between the harmonic measure $\omega$ and dynamics can be quantified.
Theorem~3 of~\cite{fine} describes the similarity between ${\cal M}_d$ and ${\cal J}_{c_0}$  through one-parameter family of asymptotically conformal maps $\Upsilon_{c_0}:\C\mapsto \C$, with $c_0$ typical  with respect to the harmonic measure
on $\partial \M$. We state it as Fact~\ref{theo:basic}. A certain complexity in the formulation of Fact~\ref{theo:basic} is related to the introduction of a full compact $Z$ which does not have a canonical  dynamical meaning. The role of $Z$ is to "enlarge" ${\cal J}_{c_0}$ to compensate for the fact that $\M$ and ${\cal J}_{c_0}$ have different topological properties,  $\M$ has a non-empty and dense interior~\cite{mss} while the corresponding Julia set ${\cal J}_{c_0}$ is a dendride. The compact $Z$ depends on a construction
as the critical orbit $\{f^n_{c_0}(c_0)\}_{n\in \NN}$ is dense in ${\cal J}_{c_0}$ and  lacks any 
$c$-stable hyperbolic structure~\cite{grsm}. Since outside $Z$, the similarity map $\Upsilon_{c_0}$ agrees with the natural univalent map 
$\Psi\circ\Psi^{-1}_{c_0}:\hat{\C}\setminus {\cal J}_{c_0}\mapsto \hat{\C}\setminus  \M$, $Z$ can be considered as an asymptotically negligible  correction of ${\cal J}_{c_0}$ near $c_0$ so that 
$\Psi\circ\Psi^{-1}_{c_0}$ extends accross $Z$ to a global quasiconformal map that fails to be analytic on $Z$ but its distortion can be still  controlled through quasiconformal constants.
\begin{fact}\label{theo:basic}
for almost every $c_0\in \partial {\cal M}_{d}$ with respect to the
harmonic measure there exist
a full compact $Z$, $c_0\in \partial Z$, a Jordan disk $U\ni c_0$,
 and a quasi-conformal map $\Upsilon_{c_0}$ of the plane, $\Upsilon_{c_0}(c_0)=c_0$,
 with the following properties: 
\begin{description}
\item{\rm (i)} 
 ${\cal J}_{c_0}\cap Z$ is connected, 
${\cal J}_{c_0}\cap \overline{U}={\cal J}_{c_0}\cap Z$, and every $z\in \partial Z\setminus \{c_0\}$ is non-recurrent,
\item{\rm (ii)}   $\lim_{r\rightarrow 0}
\frac{1}{r}d_{H}(Z\cap \D(c_0,r), {\cal J}_{c_0}\cap \D(c_0,r))=0$, where $d_H$ stands for the Hausdorff distance,
\item{\rm (iii)}  $\lim_{r\rightarrow 0}\frac{1}{r^{2}}\;
\area\left(Z \cap \D(c_0,r)\right)=0$,
\item{\rm (iv)} 
$\Upsilon_{c_0}(Z\cap U)\supset {\cal M}_{d}\cap \Upsilon_{c_0}(U)$, $Z$ is disjoint with the hyperbolic geodesic $\gamma\subset \C \setminus {\cal J}_{c_0}$ landing at  $c_0$, and  { $\lim_{\gamma\ni \xi \rightarrow c_0}d_{H}(\xi, {\cal J}_{c_0})/d_{H}(\xi,Z)=1,$}
\item{\rm (v)} $\Upsilon_{c_0}$ on $U\setminus Z$ is equal to $\Psi\circ \Psi_{c_0}^{-1}$ where $\Psi_{c_0}$ and $\Psi$ are uniforming maps from $\{|z|>1\}$
on $\hat{\C}\setminus {\cal J}_{c_0}$
and $\hat{\C}\setminus \M$, respectively, tangent to the 
identity at $\infty$,
\item{\rm (vi)} the maximal dilation of $\Upsilon_{c_0}$ restricted to $\D(c_0,r)$
tends to $1$ when  $r$ tends to $0$.
\item{\rm (vii)} $\Upsilon_{c_0}$ is conformal at $c_0$. 
\end{description}
\end{fact} 
The only claim of Fact~\ref{theo:basic} which is not contained in Theorem~3 of~\cite{fine} is the limit in {\rm (iv)}.  A short proof of {\rm (iv)} is  
delegated to Appendix.

Recall that a quasi-conformal mapping $\Upsilon$ is $(1+\beta)$-conformal at $z_0$, $\beta\geq 0$ if 
\[\Upsilon(z)= \Upsilon(z_0)+ \Upsilon'(z_0)(z-z_0)+\epsilon(|z-z_0|)\;,\]
with $\Upsilon'(z_0)\not = 0$ and $\lim_{z\rightarrow z_0} \frac{\epsilon(|z-z_0|)}{|z-z_0|^{1+\beta}}=0$.
The $1$-conformal map is called  conformal.

The proof of conformality of $\Upsilon$  in ~\cite{fine} was based on   an integral condition of Teichm\"uller, Wittich, and Belinski\u{\i}. 
If the Beltrami coefficient   $\mu(z)=\Upsilon_{\overline{z}}(z)/\Upsilon_{z}(z)$  around $c$ satisfies 
\begin{equation}\label{equ:belinski}
\int_{\DD(c,r)} \frac{|\mu(z)|}{|z-c|^2} ~dxdy < \infty
\end{equation}
for some positive $r$, then $\Upsilon$ is conformal at $c$, see~\cite{lehvi}.

\paragraph{Smooth continuity of similarity map along hyperbolic geodesics.}
Let $\gamma$ denote the hyperbolic geodesic of $\C\setminus {\cal J}_{c_0}$ which
lands at $c_0$, $Z$ denotes the continuum from Fact~\ref{theo:basic} and $\chi_Z$ is the indicator function.
Theorem~\ref{theo:30xa,1} states that there exist a bound $o(R)$, $\lim_{R\rightarrow 0^+} o(R)=0$, and $R_0>0$ such
that for every $z_0 \in \gamma$ and if $|z_0-c_0|<R_0$
\begin{equation}\label{equ:int}
 \int_{D(z_0,R)} \frac{\chi_Z(w)\,}{|z_0-w|^2} ~d\lambda_2(w)\leq
o(R) \;, 
\end{equation}
where $\lambda_2$ is $2$-dimensional Lebesgue measure.

We will need the following version of uniform conformality proved by V. Gutlyanski\u{\i} and O. Martio in~\cite{gutmar} (Theorem~1.4).
\begin{fact}\label{fa:guma}
Let $F$ be a quasiconformal self-mapping of the complex plane
with complex dilatation $\mu$ and let ${\cal K}\subset \C$ be a compact set. If there are positive
constants $R$ and $M$ such that
$$\int_{|z-w|<R} \frac{|\mu(w)|^{2}}{|z-w|^2} ~ d\lambda_2(w) \leq  M$$
holds for every $z \in  {\cal K}$  and there exists a finite limit
$$ \lim_{r\mapsto 0} \int_{r<|z-w|<R} \frac{\mu(w)}{(z-w)^2} ~ d\lambda_2(w)$$
uniformly for $z \in {\cal K}$, then the mapping $F$ is
conformally differentiable on ${\cal K}$
and the complex derivative of $f'(z)$ is continuous on $\cal K$.
\end{fact}

Combining the estimate (\ref{equ:int}) and Fact~\ref{fa:guma}, we obtain a version of uniform similarity along hyperbolic geodesics.

\begin{theo}\label{theo:continuity}
The derivative  $D_z\Upsilon_{c_0}(z)$ of the similarity map of
Fact~\ref{theo:basic} is continuous along the geodesic of
$\hat{\CC}\setminus J_{c_0}$ landing at $c_0$ for a typical point
$c_0\in \partial\M$ with respect to the  harmonic measure $\omega$.
\end{theo}
We will discuss below three applications of Theorem~\ref{theo:continuity} and the similarity structures to  some known open problems in complex dynamics.

\subsection{Deep points}

By~\cite{grsw,sm,grsm}, $\HD({\cal J}_{c_0})<2$ for almost all $c_0\in M$ with respect to the harmonic measure.
Let $\alpha=2-\HD({\cal J}_{c_0})$. Theorem~\ref{theo:24xp,1} states that for any $\epsilon>0$, 
\begin{equation}\label{equ:deep}
\lim_{r\rightarrow 0^+} \frac{|Z\cap
    \D(c_0,r)|}{r^{2+\frac{\alpha}{d}-\epsilon}}=0 
   ,
   \end{equation}
where $|\cdot|$ stands for $2$-dimensional Lebesgue measure. 
Since the similarity map $\Upsilon_{c_0}(z)$ is conformal at $z=c_0$ and $\Upsilon_{c_0}(Z)\subset \M$, we can transport the estimate (\ref{equ:deep}) to the parameter space proving that a generic parameter $c_0\in \partial \M$ with respect to the harmonic measure is $(\frac{\alpha}{d} - \epsilon)$-measurably deep with respect to
$\C\setminus \M$ according to the definition of~\cite{mcm}.
\begin{theo}\label{theo:deep}
For generic parameter $c_0$ with respect to the harmonic measure on $\partial \M$ and every  $\epsilon>0$, 
$$
\lim_{c \rightarrow c_0} \frac{\left|\M\cap
    \D(c_0,|c-c_0|)\right|}{|c_0-c|^{2+\frac{\alpha}{d}-\epsilon}}=0.
  $$ 
\end{theo}
Another consequence of the estimate (\ref{equ:deep}) is an improved integrability in (\ref{equ:belinski}). Theorem~2.25 in \cite{mcm} asserts  that if $c_0$ is a
$\delta$-deep point of $\C\setminus \M$ and $\Upsilon:\C\mapsto\C$  $K$-quasiconformal map then $\Upsilon$ is $(1+\beta(\delta,K))$-conformal at $c_0$. The dilatation
of the similarity map $\Upsilon$ tends to $1$ when $c$ approaches $c_0$.
\begin{coro}\label{theo:conf}
The similarity map $\Upsilon$ is  $(1+\frac{\alpha}{d}-\epsilon)$-conformal for every $\epsilon>0$ and almost all  $c_0\in \partial \M$ with respect to the harmonic
measure.
\end{coro}

Theorem~\ref{theo:deep} and Corollary~\ref{theo:conf} are generalizations of the results of \cite{rivera} 
obtained  for non-recurrent parameters (Misiurewicz case). Note that Misiurewicz set of parameters is of
harmonic measure $0$, see~\cite{grsw,sm}. The concept of measurable deep points was proposed  by C. McMullen in the context of renormalization~\cite{mcm}. 

The proof of Theorem~\ref{theo:deep} is based on a global inductive estimate of conformal densities distributed over elements of Yoccoz partitions. The main technical difficulty 
to overcome is a tendency of conformal measures to excessively concentrate  around recurrent critical points~\cite{grsm,rlporous}. 

A finite Borel measure $\nu$ supported on $\J_c$ 
is called {\em conformal with an exponent} $\kappa>0$
(or {\em $\kappa$-conformal}) if for every Borel set
$B$ on which $f$ is injective one has
\[ \nu(f(B))=\int_{B}|f_c'(z)|^{\kappa}~d\nu(z)~.\]
Of particular importance are conformal measures with  the minimal exponents,~\cite{sulcio, du}. In~\cite{grsm1}
it was proved that for a large class of rational maps, including Collet-Eckmann quadratic polynomials, conformal measures with the minimal exponent $\kappa$ are ergodic (hence unique), non-atomic, and 
$$\kappa=\HD(\J_c) =\HD(\nu):=\inf_{A:\nu(A)=1}\HD(A).$$

\subsection{Transversality function}
 M. Benedicks and L. Carleson  in their work on unimodal maps $z^2+c$, $c\in \RR$, and the H\'enon map, \cite{Beca,Becaa},  
used the transversality function 
\begin{equation}\label{equ:26mp,1}
  {\T}(c)=\sum_{n=0}^{\infty}\Bigl(D_zf_{c}^n(z)_{z=c}\Bigr)^{-1}
\end{equation}  
to control distortion between the phase and parameter spaces. It was observed in ~\cite{Beca, Becaa} that 
as long as ${\cal T}(c)\not = 0$ and ${\cal T}(c)\not=0$ and $|{\T}|(c)=\sum_{n=0}^{\infty} |D(f_{c}^n)(c)|^{-1} < \infty$ then the parameter exclusion  construction can be initiated. The outcome
of the construction is a set of parameters of positive Lebesgue measure with  an expanding dynamics. 
The work ~\cite{Beca, Becaa} generalized an earlier breakthrough due to  M. Jakobson on the existence of a set of parameters of 
positive $1$-dimensional Lebesgue measure with a stochastic dynamics.  The proof of M. Jakobson was based on very different
techniques than that of~\cite{Beca,Becaa}. 
\begin{theo}\label{theo:trans1}
The sum
 $$ {\T}(c_0)=\sum_{n=0}^{\infty}\Bigl(D_zf_{c_0}^n(z)_{z=c_0}\Bigr)^{-1}$$
converges for almost all $c_0\in \partial \M$ with respect to the
harmonic measure and  satisfies
$${\T}(c_0)=\frac{1}{D_z\Upsilon_{c_0}(z)_{|z=c_0}}\;,$$
where $\Upsilon_{c_0}(z)$ is the similarity function of Fact~\ref{theo:basic}.
\end{theo}
For Misiurewicz parameters, Theorem~\ref{theo:trans1} was proven by J. Rivera-Letelier in ~\cite{rivera}. The proof 
in~\cite{rivera} is based on transversality of two different holomorphic motions, the critical value $f_c(c)$ and the postcritical hyperbolic 
compact ${\cal P}(c)$ for $c$ from a small neighborhood of $c_0$. Our proof is different as dynamics generic with respect to the harmonic measure does not
have an underlying  hyperbolic structure. The main idea is to produce  uniform estimates for $\T(c)$ outside of $\M$ at some scales 
and then pass to the limit  along the hyperbolic geodesic landing at at $c_0$. The main technical ingredient is 
$C^1$-smoothness of the similarity map  $\Upsilon_{c_0}(z)$ along hyperbolic geodesics as stated in Theorem~\ref{theo:continuity}.

Fact~\ref{fa:1kp,1} states that for $c\not \in \M$, ${\cal T}(c)=0$ iff $D_c\Psi(c)=0$. Since $\Psi(c)$ is univalent, ${\cal T}(c)$ can not vanish. This, however, is not a new result, since it follows from a somewhat more general theorem of~\cite{lev}.

The transversality condition is closely related to the summability conditions in complex dynamics, 
$|{\T}|^{\beta}(c)=\sum_{n=0}^{\infty}|D(f_{c}^n)(c)|^{-\beta} <+\infty$,
$\beta\in (0,1]$, which imply various degrees of metrical or conformal smallness of  ${\cal J}_c$ ~\cite{grsm1, brstr, rish, lev}.

\paragraph{Geometric interpretation of the transversality function.} For a typical $c_0\in \partial \M$ with respect to the harmonic measure, the similarity function 
$\Upsilon_{c_0}$ maps a hyperbolic geodesic $\gamma$ of $\C\setminus {\cal J}_{c_0}$ landing
at $c_0$ onto a hyperbolic geodesic $\Gamma$  of $\C\setminus \M$ landing at $c_0=\Upsilon_{c_0}(c_0)$, see Fact~\ref{theo:basic} (\rm{v}). Let $\gamma(z)$ denote the subarc of $\gamma$ between $z\in \gamma$ and $c_0$ and $|\gamma(z)|$ be the length of $\gamma(z)$.

Theorem~\ref{theo:trans1} 
and $C^1$-smoothness of $\Upsilon_{c_0}$ along $\gamma$ yield the following corollary.
\begin{coro}\label{geo}
For almost all $c_0\in \partial \M$ with the respect to the harmonic measure,
\begin{equation}\label{equ:modul}
|{\cal T}(c_0)|=\lim_{c\in \Gamma \rightarrow c_0}\frac{|\gamma(\Upsilon_{c_0}^{-1}(c))|}{|\Gamma(c)|}\;
\end{equation}
\end{coro} 

A formula similar to (\ref{equ:modul}) holds for $\arg {\cal T}(c_0)$ but its dynamical meaning seems to be less clear
within Benedicks-Carleson perturbation theory. Also, according to~\cite{fine}, the limit $\lim_{c\in \Gamma \rightarrow c_0} \arg (\Gamma(c))-c_0)$ does not exist for almost all $c_0\in \partial {\cal M}_d$ with respect to the harmonic measure as $\Gamma$ twists around $c_0$ in both directions infinitely many times. In~\cite{sm}, it was proved that the Collet-Eckmann condition holds for all $c\in \partial \M$ except possibly for  a set of harmonic Hausdorff dimension $0$. One can ask if the formula $(\ref{equ:modul})$ holds for all Collet-Eckmann parameters or even for the summability class $|{\T}|^{1}(c)<\infty$.




\subsection{Geometric applications}
\paragraph{Flat angles.} 
Let $\K=\partial K$ be a continuum. 
 $\K$ is well-accessible at $y\in \K$ (or accessible within a twisted angle)  if there exist a Jordan curve 
$\gamma\subset\C\setminus K$ terminating at $y$ and $C>0$ such that for every $z\in \gamma$,
$$\dist(z,{\cal K})> C\;\diam \gamma(z),$$
where $\gamma(z)$ is the subarc of $\gamma$ between $z$ and $y$. 
If every point from $\K$ is 
accessible within a twisted angle of the same aperture then $\C\setminus \K$ is a John domain.
If $y$ is well-accessible then it is also well-accessible by the hyperbolic geodesic landing at $y$ \cite{pom}. 
Theorem~3 of~\cite{fine} states that
for almost every $c\in \partial {\cal M}_{d}$ with respect to the
harmonic measure $\omega$, the parameter $c$ is a Lebesgue density point
of $\C\setminus {\cal M}_{d}$ but it is not well-accessible.

We say that a point $c^*\in \partial \M$ is iterated $\log$-accessible if a hyperbolic geodesic $\Gamma$ lands at $c^*$ and for any $m>0$,

$$\lim_{ \Gamma\ni c \rightarrow c^*}
\frac{\dist(c,{\cal M}_d)}{\diam{\Gamma}(c)} \log_{[m]} \frac{1}{\diam{\Gamma}(c)} 
=+\infty,$$
where
$\log_{[m]}=\log\circ\dots\circ \log$ is the $m$-th iterate of $\log$ function.

\begin{theo}\label{theo:flat}
For almost every $c^*\in \partial {\cal M}_{d}$ with respect to the
harmonic measure,
$c^*$ is iterated $\log$-accessible.
\end{theo}
Theorem~\ref{theo:flat} follows from the existence of the similarity structures and  an  iterated  large deviation estimate for exponential distribution, see~\cite{grasw} for a detailed proof. 
\paragraph{Hedgehogs and porosity in the parameter space.}
 The concept  of porosity has a long history, see~\cite{mat2}. A set $E\subset \C$
is $\beta$-{\em porous}, $\beta>0$, at $z^*\in E$ and scale
$r>0$ if there is $z\in \D(z^*,r)$ such that $\D(z,\beta r) \cap  E = \emptyset$. 

By the Makarov law of the iterated logarithm~\cite{mak}, almost every point from $\partial \M$ with respect to the harmonic measure  is H\"older accessible, and thus $\beta$-porous, $\beta\in (0,1/2)$ in many scales, see Proposition~2.2 in ~\cite{gsarkiv}. 
The limiting value of $\beta=\frac{1}{2}$ from~\cite{gsarkiv} falls short of the upper bound $1$. It is  not know what happens for $\beta$ between $1/2$ and $1$.

The harmonic measure is supported on a set of points of $\partial \M$ that can only be accessed by passing through infinitely many increasingly narrow "tunnels" at scales of positive density.
The prelevance of such extremal sets in complex dynamics was shown in ~\cite{gsarkiv}. Using the similarity structures from Fact~\ref{theo:basic}, one can quantify the lack of porosity and prove that for a  typical $c^*\in \partial \M$ with respect to the harmonic measure, accessibility within a John angle fails rather badly and 
an extremal ``non-accesibility'' in the sense of  Makarov theory~\cite{mak} is observed instead~\cite{grasw}.

\begin{figure}[tp] \label{fig:18jp,1}
\hspace{3cm}
\includegraphics[width=10cm]{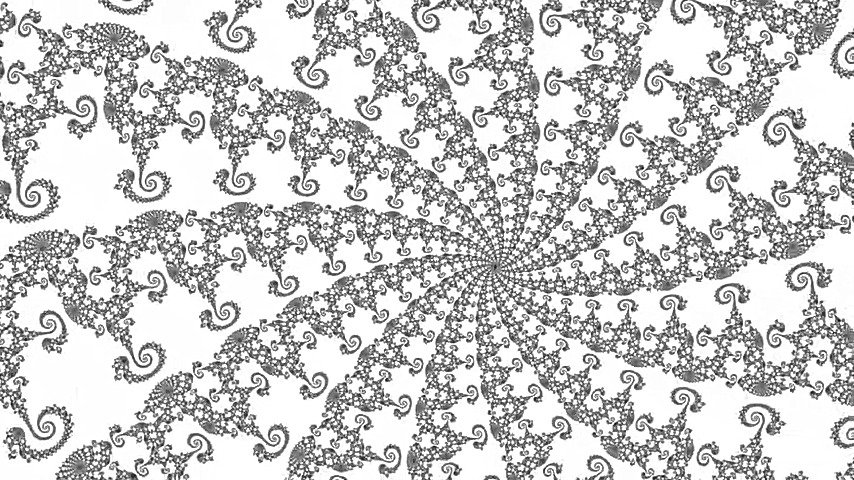}  
 \caption{\small {Hedghog layer at a typical point in the boundary of the Mandelbrot set}}
\end{figure}

We will illustrate some of these extremal  features of the harmonic measure distribution on $\partial \M$, see Figure~\ref{fig:18jp,1}. To this aim we will need a concept of hedgehog neigborhoods.
 
Let $X$ be a planar set. We say that $X$ contains {\em $(m,\epsilon)$-hedgehog layer} around  $x\in X$ if there exist a ring domain $A$,\hspace{-2,5mm} $\mod A\geq m$, and a collection of pairwise disjoint continua ${\cal C}_{k}\subset X$, $k=1,2 \dots$, with the property that
$\rm {(i)}$ $x$ belongs to the bounded component of $\C\setminus \overline{A}$, $\rm {(ii)}$ 
every  ${\cal C}_k$ intersects both components of $\C\setminus \overline{A}$, $\rm {(iii)}$ every point from $A$ is at the distance at most $\epsilon/\diam A$ to some ${\cal C}_k$ from the collection. 

Even though not explicitely stated, the concept  of hedgehog layers was introduced by J. Riviera-Letelier in his study of porosity at critical recurrent points for rational functions, see the proof of Theorem~C' in~(\cite{rlporous}). 


 We say that $X$ has {\em hedgehog neighborhood} at $x$ if for every $\epsilon, m >0$ there exists an $(m,\epsilon)$-hedgehog layer around $x\in X$. The phase-parameter space similarity of Fact~\ref{theo:basic} allows to detect  hedgehog neighborhoods in the parameter space.
\begin{theo}\label{theo:hedgehog}
The boundary $\partial \M$ contains  hedgehog neighborhood at almost every point  $c^*\in \partial \M$ with respect to the harmonic measure. The corresponding Julia set ${\cal J}_{c^*}$ has hedgehog neighborhoods at a dense subset of
${\cal J}_{c^*}$.
\end{theo}

Hedgehog neighborhoods are directly  related to the concept of "hairiness" proposed by J. Milnor in the context of renormalization. 
Theorem~\ref{theo:hedgehog} indicates that increasingly dense parts of the boundary of the  Mandelbrot set is a standard  feature of recurrent and non-linear dynamics rather than a staple of the renomalization.

The proof of Theorem~\ref{theo:hedgehog} explains how the construction of hedgehog neighborhoods in the phase space falls naturally into the setting of box mappings~\cite{hipek} in the unicritical case. Since hedgehog neighborhoods are quasionformal invariants, their abundance in the boundary of the Mandelbrot set follows directly from  Fact~\ref{theo:basic}. 

\section{Constructions}
\subsection{Preliminaries and the similarity map.}\label{sec:25hp,1}
We will follow closely the definitions and notations of~\cite{fine}.
Here is a partial list.
\begin{itemize}
\item $f_c(z) = z^d + c$, where $d>1$ is fixed, ${\cal J}_c$ is its
  Julia set, ${\cal K}_c$ the filled-in Julia set. 
\item
 ${\cal M}_d$ is the locus of connectivity of the family
  $\{f_c\}_{c\in\CC}$.
\item
$\Psi$ is the Riemann map from the complement of $\overline{\D(0,1)}$
  onto the complement of ${\cal M}_d$ tangent to the dentity at $\infty$; analogously, $\Psi_c$ is the
  Riemann map of the complement of ${\cal K}_c$ if $c\in {\cal M}_d$,
  otherwise $\Psi_c$ can be defined as the B\"{o}ttker coordinate on a
  neighborhood of $\infty$ and extended by the dynamics till the Green line $G_c(0)$, 
  $$G_c(z)=\lim_{n\rightarrow \infty}\frac{\log f_c^n(z)}{d^n}\;.$$
There is an explicite formula, $\Psi^{-1}_c(z)=\exp(G_c(z)+2\pi i \theta)$,
$G_c(z)>G_c(0)$, where $\theta \in [0,1)$ is called {\em external argument} 
or {\em external angle} of $z$~\cite{caga,etiuda}
  \end{itemize}
\paragraph{Rays, geodesics, and external angles.}
When $c\notin {\K}_{c}$ then the Green function $G_c$ has critical points at 
$f_c^{-i}(0)$ for $i = 0, 1,\cdots $. A smooth ray in the phase space is a gradient line of the $G_c$ with closure that intersects both $\infty$  and ${\K}_c$. We will consider only  gradient lines which avoid critical points of $G_c$ and are, therefore, smooth.  The closure of some rays  intersects ${\K}_c$ at precisely one point. We say that these rays land at (or converge to)  that point. All gradient lines are well defined on the set $\{z : G_c(z) > G_c(0)\}$.  They are labeled by the external angles $\theta\in [0,1)$ at which they enter $\infty$. 
If ${\K}_c$ is connected then the ray 
$\gamma_{\theta,c} $ with an external argument $\theta$ is a hyperbolic geodesic in $\hat{\C}\setminus {\K}_c$.

Of particular importance is the critical external angle $\theta(c)$, the angle of the gradient line which passes through $c$. Any line in the parameter space of the form $\theta(c) = \omega$ will be named an external ray with angle $\omega$ and denoted  by $\Gamma_{\omega}$ or simply $\Gamma$.
The following relation holds,
$$c\in \Gamma_{\omega}\Leftrightarrow c\in \gamma_{\omega,c}
\;.$$
 The external rays are hyperbolic geodesics in $\hat{\C\setminus \M}$. The Green function for $\M$ satisfies  $G_{\M}(c)=G_c(c)$ and for every $c\in \C\setminus {\M}$,
$$\Psi^{-1}(c)=\exp(G_c(c)+2\pi i \theta(c))\;.$$

\paragraph{Yoccoz puzzle pieces.}
Again, we refer to the construction in~\cite{fine}.
An initial order $0$ Yoccoz puzzle  is regarded as fixed and then a Yoccoz
puzzle piece of order $k\geq 0$ is one that is mapped into a piece of
order $0$ by $k$ iterations.

$b_{k,c}$ will denote a piece of order $k$
which contains $0$ - it may not exist for all $k$. Then $\beta_{k,c}=f_c(b_{k,c})$.
Since $c$ and $0$ are in different pieces of order $0$, $\beta_{k,c}$ is
disjoint from any piece which contains $0$. 

\paragraph{Nesting for typical parameters.}
Fix a typical parameter $c_0$ with respect to the harmonic measure. By
Proposition 8 of~\cite{fine}, for any $M^*$ we can find a sequence of nesting
critical pieces
\[ b_{N_0,c_0} \supset b_{N_1,c_0} \supset b_{N_2,c_0} \supset b_{N_3,c_0} \supset b_{N_4,c_0}\]
and a box locus $V_{N_5}$, $N_5>N_4>\cdots>N_1>10$, such that for every $c\in V_{N_5}$ the
nesting condition mentioned above also holds, and
\[ \mod \left( b_{N_{j}} \setminus \overline{b}_{N_{j-1}} \right) \geq
M^* \; \]
for $j=1,2,3,4$.

$M^*$ is a parameter of the construction which in turn defines $N_j$,
$j=0,1,\cdots,5$. For brevity, write $Q(M^*)$ for constants which only depend on
$M^*,d,N_j$.
When dynamical objects depend on $c$, we will supress $c_0$ from 
the notation, i.e. $b_{N_1,c_0}$ could simply be $b_{N_1}$.

\paragraph{Returns to a large scale.}
For a typical $c_0$ we can futher construct an increasing sequence $(S_n)_{n\geq 1}$ such that
\begin{itemize}
\item
  for every $n$, there is a critical piece $b_{S_n+N_0}$ which is
  mapped uni-critically onto $b_{N_0}$ by $f^{S_n}$, 
\item
  for every $n$, $f^{S_n}(0)\in b_{N_4}$,
\item
  $S_n > 10 N_4$ and $\frac{S_{n+1}}{S_n} < \frac{11}{10}$ for all $n$,
\item
  $ \lim_{n\rightarrow\infty} (S_{n+1} - S_n) = \infty$,  
   \item $ \lim_{n\rightarrow\infty}\frac{S_{n+1}}{S_n} =1\;. $
\end{itemize}

\paragraph{First return maps.}
If $b_{n,c}$ is a critical piece, then $\phi_{n,c}$ will denote the
first entry map into $b_{n,c}$ (first entry meaning that it is the
identity on $b_{n,c}$ itself). 

Let $\Phi(c,z)$ denote the natural holomorphic motion, wherever it is
defined. 
\begin{lem}\label{lem:22xa,1}
For any $n$ and $c\in V_{S_n+N_1}$, the natural holomorphic motion
starting at $c_0$ is defined on the complement of the closure of the domain of
$\phi_{S_n+N_1}$.
\end{lem}
\begin{proof}
  Take a point $z_0$ in the complement of the closure of the domain of
  $\phi_{S_n+N_1}$. We will show that its natural holomorphic motion
  extends to $V_{S_n+N_1}$. By Lemma 2.6 of~\cite{fine}, we know that
  $V_{S_n+N_1}$ is simply connected. It will suffice to prove that for
  any quasi-disk $D$ compactly contained in $V_{S_n+N_1}$, the holomorphic
  motion can be extended to an open set which contains
  $\overline{D}$. 

  For any $c\in D$, the orbit of $\Phi(c,z_0)$ under $f_c$ forever avoids
  $b_{S_n+N_1,c}$. From Lemma 2.8 of~\cite{fine}, it implies that the
  distance of that orbit to $0$ remains uniformly bounded way from $0$
  on $D$. By Lemma 2.2 of~\cite{fine}, if $c'$ is now on the boudnary
  of $D$, that means that $\Phi(c,z_0)$ extends to a neighborhood of
  $c'$.   
\end{proof}

\begin{lem}\label{lem:23xa,1}
  There is a natural holomorphic motion defined on
  \[ V_{S_n+N_1} \times \left( \partial\beta_{S_n+N_2} \setminus {\cal J}\right) \]
\end{lem}
\begin{proof}
 By Lemma 2.10 of~\cite{fine} we need to check that $f^k(\partial
 b_{S_n+N_2}) \cap b_{S_n+N_1} = \emptyset$ for $0<k\leq N_2-N_1$. If
 for any $k>0$ that intersection is non-empty, then $f^k(b_{S_n+N_1})
 \supset b_{S_n+N_1}$. Since $f^{S_n}$ is uni-critical on
 $b_{S_n+N_1}$ the smallest $k$ for which it could occur is $k=S_n$.   
 But we assumed $S_n > N_5 > N_2$.
\end{proof}

\begin{lem}\label{lem:23xa,2}
When $\Phi$ is the holomorphic motion on $\partial \beta_{S_n+N_2}$,
natural outside of $\cal J$, then the equation
\[ \Phi(c,z) = c \]
for $z\in\partial\beta_{S_n+N_2}$ has exactly one simple zero on
$\partial V_{S_n+N_2}$. 
\end{lem}
\begin{proof}
  That follows directly from Lemma 2.11 in~\cite{fine}.
\end{proof}

The key fact which establishes the existence of the similarity map is
the following:

\begin{prop}\label{prop:23xp,1}
There is a $Q(M^*)$- quasiconformal homeomorphism $\Upsilon_n$ defined on a
neighborhood of the $V_{S_n+N_3}$ which fixes $c_0$ and coincides with
$\Psi\circ\Psi_{c_0}^{-1}$ outside the closure of the domain of
$\phi_{S_n+N_1}$. The constant $Q(M^*)$ tends to $1$ as $M^*$ tends to
$\infty$. 
\end{prop}

This follows from Proposition 5 of ~\cite{fine}, while the modulus
claim follows from Lemma 4.3. 

\paragraph{The similarity map.}
The similarity map $\Upsilon_n$
allows one to subdivide $V_{S_n+N_3}$ in a way that
is homeomorphic to the subdivision of $b_{S_n+N_1}$ into the components
of the domain of $\phi_{S_{n+1}+N_1}$. This subdivision is the best we can
do on the annulus $V_{S_n+N_3}\setminus \overline{V}_{S_{n+1}+N_3}$
since the inner component can then be subdivied using $\Upsilon_{n+1}$
and they will match along the common boundary.

Let $\Upsilon$ mean the homeomorphism defined on a neighborhood of $c$
which is $\Upsilon_n$ on $A_n$. It is quasi-conformal, since the
boundaries of pieces $\beta_{S_n+N_3,c}$ are quasi-circles and
therefore removable.

\subsection{Nesting of Yoccoz pieces.}
\begin{lem}\label{lem:24xa,1}
For any $n$, pieces $f^j(b_{S_n+N_1})$ are disjoint from $b_{S_n+N_1}$
for $1\leq j<S_n$.
\end{lem}
\begin{proof}
Since $f^{S_n}$ is uni-critical on $b_{S_n+N_1}$, we cannot have
$f^j(b_{S_n+N_1}) \supset b_{S_n+N_1}$. The opposite inclusion is also
impossible, because eventually $b_{S_n+N_1}$ must be mapped on a piece
of order $0$.
\end{proof}

\paragraph{The predecessor function.}
\begin{defi}\label{defi:28xa,1}
  For $n\geq 1$, let $\sigma(n)$ denote the smallest $k\geq 1$ for which
  $S_k \geq \frac{S_{n+2}}{2}$.
\end{defi}

By our hypothesis, for $n>1$ we get $\sigma(n)\leq n-1$. 

\begin{lem}\label{lem:23xp,1}
Let $n\geq 2$ and $b_{k}$ denote the critical component of the domain
of the first return map into $b_{S_{\sigma(n)}+N_1}$. Then $k>S_{n+1}+N_4$.
\end{lem}
\begin{proof}
Since $b_k \subset b_{{S_\sigma(n)}+N_1}$, then we must have $k\geq S_{\sigma(n)}$ by Lemma~\ref{lem:24xa,1}.

Suppose next that $b_k$ contains $b_{S_{n+1}+N_1}$. Then  $S_{n+1} \geq k+S_{\sigma(n)}$, since $f^k$ first
maps $b_{S_{n+1}+N_1}$ into
$b_{S_{\sigma(n)}+N_1}$ which needs at least $S_{\sigma(n)}$ more iterates to cover $b_{N_1}$, while $f^{S_{n+1}}(b_{S_{n+1}+N_1})=b_{N_1}$.
Hence, $S_{n+2} > S_{n+1} \geq 2S_{\sigma(n)}$,
which contradicts Definition~\ref{defi:28xa,1}.

Consequently, $b_k$ is strictly contained in $b_{S_{n+1}+N_1}$. 
Since $0\in b_k$ is mapped by $f^{S_{n+1}}$
into $b_{N_4}$, $f^{S_{n+1}}(b_k) \cap b_{N_4} \neq \emptyset$. If
$f^{S_{n+1}}(b_k) \supset b_{N_4}$, then recall that $b_k$ is a domain of the first
retun map into $b_{\sigma(n)}$. Since we assumed $S_n \geq 10N_4$ for all $n$,
$b_{\sigma(n)} \subsetneq b_{N_4}$.
So, $f^p(b_k)$ must have covered
$b_{\sigma(n)}$ including $0$ for some $0<p<S_{n+1}$, but this contradicts Lemma~\ref{lem:24xa,1}.

The only remaining possiblity is $f^{S_{n+1}}(b_k) \subsetneq b_{N_4}$, in which
case $k > S_{n+1}+N_4$.
\end{proof}

We will write $A_n =
\beta_{S_n+N_3}\setminus\overline{\beta}_{S_{n+1}+N_3}$. 

\begin{lem}\label{lem:25xp,1}
Any component of the domain of $\phi_{S_{\sigma(n)}+N_1}$ which intersects
$A_n$ is contained in it.
\end{lem}
\begin{proof}
  Let $\tilde{\zeta}$ be a component of the domain of $\phi_{S_{\sigma(n)+N_1}}$.
  Since Yoccoz pieces intersect only if one contains the other, the claim of
  the Lemma is equivalent to showing that $\beta_{S_{n+1}+N_3}\not\subset\tilde{\zeta}$. If, to the contrary, the inclusion holds, then $f^{-1}(\tilde{\zeta})$ contains a critical piece which is a component
of the domain of the first return map into $b_{S_{\sigma(n)}+N_1}$. That
piece cannot contain $b_{S_{n+1}+N_3}$ by Lemma~\ref{lem:23xp,1}.
\end{proof}  
  
\begin{prop}\label{prop:24xa,1}
  On any component of its domain,
  the mapping $\phi_{S_n+N_3}$ extends univalently to
range $b_{S_{\sigma(n)+N_1}}$. 
\end{prop}
\begin{proof}
Write $\zeta$ for the component of the domain of $\phi_{S_n+N_3}$ and
let $\tilde{\zeta}$ be the component of the domain of $\phi_{S_{\sigma(n)}+N_1}$
which contains $\zeta$.

For some $k$, $f^k$ maps $\tilde{\zeta}$ univalently onto
$b_{S_{\sigma(n)}+N_1}$ and $\zeta$ into a subpiece $f^k(\zeta)$. If $0\in f^k(\zeta)$,
then since $\zeta$ was a component of the first entry map into $b_{S_n+N_3}$,  $f^k(\zeta)$ coincides with $b_{S_n+N_3}$ and the claim of the Proposition follows.

Otherwise, $f^k(\zeta) \cap b_{S_n+N_3} =\emptyset$. Then
consider the first return map from $b_{S_{\sigma(n)}+N_1}$ into
itself. $f^k(\zeta)$ belongs to some component $\tilde{\zeta}_1$
of the domain of that map. It
cannot be the critical component which must be contained in
$b_{S_{n+1}+N_3}$ by Lemma~\ref{lem:23xp,1}. Thus, $\tilde{\zeta}_1$ is mapped onto $b_{S_{\sigma(n)}+N_1}$
univalently by some $f^{k_1}$ and $f^{k+k_1}(\zeta)$ is again a subpiece
of $b_{S_{\sigma(n)}+N_1}$. Then we repeat the entire reasoning to conclude
  that either $f^{k+k_1}(\zeta) = b_{S_n+N_3}$ and the claim of the
  Proposition follows, or $f^{k+k_1}(\zeta)$ belongs to a non-critical
  component $\tilde{\zeta}_2$ of the first return map into
  $b_{S_{\sigma(n)}+N_1}$ and can be
  pushed univalently by another $f^{k_2}$. The process has to end
  eventually, since $k+\sum k_j$ cannot exceed the order of $\zeta$.
\end{proof}

\begin{coro}\label{coro:24xa,1}
If a component of the domain of $\phi_{S_n+N_1}$ is contained in $A$,
the domain of univalent extension onto $b_{S_{\sigma(n)}+N_1}$ mentioned in
Proposition~\ref{prop:24xa,1} is disjoint from the external ray which
lands at $c_0$.
\end{coro}
\begin{proof}
  Since that domain is a Yoccoz puzzle piece and does not contain $c_0$,
  it is disjoint from the ray.
\end{proof}

\section{Metric estimates}
\subsection{Uniform shrinking.}
\paragraph{Lyapunov exponent.}
Let $\lambda$ denote the Lyapunov exponent of $f_{c_0}$ at $c_0$. We know that
$\lambda>0$ by~\cite{grsw,sm} and furthermore, $\lambda = \log d$
by~\cite{gniotek}. 

\paragraph{Roundness of pieces.}
Let us introduce a definition.
\begin{defi}\label{defi:24xp,1}
Consider a simply connected bounded domain $U\subset \CC$ and $z_0\in U$. We will say
that $U$ is $K$-balanced with respect to $z_0$ if for any $z\in U$,
$\theta\in\RR$, $z_0 + K^{-1}e^{i\theta}(z-z_0) \in U$.  
\end{defi}  

Domains $\beta_{S_n+N_j}$, $j=1,\cdots, 4$ are $K(M^*)$ balanced with
respect to $c$, while $b_{S_n+N_j}$ are $Q(M^*)$ balanced with respect
to $0$.  
They also are $K(M^*)$-quasi-discs. These properties will be
referred to as the {\em roundness} of critical pieces.
The roundness directly follows from the conditions imposed on returns
to the large scale. 

\begin{lem}\label{lem:26xp,1}
 \[ \log \left(\diam\beta_{S_n+N_1}\right)^{-1} = S_n\lambda +
 o_{M^*}(S_n) \; .\]
 \end{lem}
\begin{proof}
  $f^{S_{n}-1}$ maps $\beta_{S_n+N_1}$ onto $b_{N_1}$ with distortion bounded
    in terms of $M^*$, since the map extends univalently onto
    $b_{N_0}$.  The estimate follows from the notion of the Lyapunov exponent. 
 \end{proof}

\begin{lem}\label{lem:25xp,2}
  For all $n>1$,
  \[ \sum_{k=2}^n \mod\left( \beta_{S_{k-1}+N_1}
  \setminus\overline{\beta}_{S_k+N_1}\right) =
   S_n\lambda + o_{M^*}(S_n) \; .\]
\end{lem}
\begin{proof}
    Since pieces $\beta_{S_n+N_1}$ are all round,
    \[ \sum_{k=2}^n \mod\left( \beta_{S_{k-1}+N_1}
      \setminus\overline{\beta}_{S_k+N_1} \right) =
      -\log\diam\beta_{S_n+N_1} + O_{M^*}(n) \; .\]
    Since $\lim_{n \rightarrow\infty} S_n/n = \infty$, the term linear
    in $n$ can be absorded into the constant $o_{M^*}(S_n)$ and so the
    Lemma follows from Lemma~\ref{lem:26xp,1}.   
\end{proof}

\begin{prop}\label{prop:25xp,1}
For every component $\zeta$ of the domain of $\phi_{S_n+N_1}$, 
\[ \log\left(\diam(\zeta)\right)^{-1}  \geq S_n \frac{\lambda}{d}
+o_{M^*}(S_n) \; .\] 
\end{prop}
\begin{proof}
  From Lemma~\ref{lem:25xp,2},
   \[ \sum_{k=2}^n \mod\left( b_{S_{k-1}+N_1}
  \setminus\overline{b}_{S_k+N_1}\right) =
   S_n\frac{\lambda}{d} + o_{M^*}(S_n) \; .\] 
From Proposition~\ref{prop:24xa,1} $\zeta$ is surrounded by nesting
annuli which are conformally equivalent to $b_{S_{k-1}+N_1}
  \setminus\overline{b}_{S_k+N_1}$.
  The claim follows by superadditivity of moduli
  and Teichm\"{u}ller's modulus estimates, see~\cite{lehvi}.  
\end{proof}

\paragraph{Additional estimates on the sizes of pieces.}
Now we denote by $\{\zeta_{n,j}\}_{j=1}^{\infty}$
the components of the domain of $\phi_{S_n+N_1}$
which are contained in $A_n$.

\begin{lem}\label{lem:26xp,2}
For any $\epsilon>0$ and $M^*$ there is $n_0$ such that if $n\geq
n_0$, then  
\[ \sup\left\{ \diam \zeta_{n,j} :\: j=1,\cdots \right\} \leq
  \diam\left(\beta_{S_n+N_3}\right)^{1+\frac{1}{d}-\epsilon} \;
  .\]
  \end{lem}
\begin{proof}
  $\beta_{S_n+N_3}$ together with any $\zeta_{n,j}$ contained in it are mapped
    by $f^{S_n-1}$ into $b_{N_3}$. By Proposition~\ref{prop:25xp,1},
    \[ \diam f^{S_n-1}(\zeta_{n,j}) \leq \exp\left( - S_n\frac{\log
      d}{d}+o_{M^*}(S_n) \right) \; .\]
    Taking into account Lemma~\ref{lem:26xp,1} 
    \begin{equation}\label{equ:26xp,1}
      \diam f^{S_n-1}(\zeta_{n,j}) \leq
      \left(\diam\beta_{S_n+N_1}\right)^{1/d} \exp(o_{M^*}(S_n)) \; .
      \end{equation}
    Again by Lemma~\ref{lem:26xp,1},
    \[ \exp(o_{M^*}(S_n)) \leq
    \left(\diam\beta_{S_n+N_1}\right)^{-\epsilon} \]
    for any $\epsilon>0$ provided that $n$ is suffciently large. 

    Pulling back by $f^{S_n-1}$ will introduce another factor
    $\diam\beta_{S_n+N_1}$ on the right-hand since together with an
    error term depnding on $M^*$, which can be be absorbed in
    $o_{M^*}(S_n)$.
  \end{proof}

  \begin{lem}\label{lem:26xp,3}
  For any $\epsilon>0$ and $M^*$ there is $n_0$ such that whenever
  $n\geq n_0$, then
  \[ \diam\beta_{S_{n+2}+N_1} \geq
  \left(\diam\beta_{S_{n-1}+N_1}\right)^{1+\epsilon}\; .\]
  \end{lem}
  \begin{proof}
  By Lemma~\ref{lem:26xp,1}
  \[ \log\frac{\diam\beta_{S_{n-1}+N_1}}{\diam\beta_{S_{n+2}+N_1}} =
  (S_{n+2}-S_{n-1})\lambda + o_{M^*}(S_{n+2})+o_{M^*}(S_{n-1}) \; .\]
  Since $\lim_{n\rightarrow\infty}\frac{S_{n+1}}{S_n} = 1$, the
  right-hand side is $o_{M^*}(S_{n-1})$, which is at least
  $(\diam\beta_{S_{n-1}+N_1})^{-\epsilon}$ provided that $n$ is large
  enough.
\end{proof}

\subsection{Estimates based on the conformal measure.}
Let $\nu$ denote the conformal measure on $\cal J$. The exponent of
$\nu$ will be denoted with $2-\alpha$ and is equal to $\HD({\cal
  J})$. The existence of  a unique  non-atomic $\nu$ with the minimal exponent $\HD({\cal
  J})$ was established in~\cite{grsm1}.
Since $\HD({\cal J}) < 2$ by~\cite{grsw}, $\alpha>0$. 

\begin{lem}\label{lem:24xp,1}
Let $\zeta$ denote a component of the domain of $\phi_{S_n+N_j}$,
$n>1$, $j=1,2,3,4$. Then
\[ \nu(\zeta) \geq K_1(M^*) \left(\diam\zeta\right)^{2-\alpha} \; .\]
\end{lem}
\begin{proof}
By Proposition~\ref{prop:24xa,1} domain $\zeta$ is mapped onto
$b_{S_n+N_j}$ with distortion which is bounded depending on $M^*$.  
Because of roundness we get
\[ \nu(\zeta) \geq L_1(M^*) \nu(b_{S_n+N_j}) \left(
\frac{\diam \zeta}{\diam b_{S_n+N_j}}\right)^{2-\alpha} \; .\]

One can further see that
\[ \nu(b_{S_n+N_j}) \geq L_2(M*)\nu(\beta_{S_n+N_j}) \left(\frac{\diam
  b_{S_n+N_j}}{\diam\beta_{S_n+N_j}}\right)^{2-\alpha} \]
and since $\beta_{S_n+N_j}$ is mapped onto $b_{N_j}$ with bounded
distortion,
\[ \nu(\beta_{S_n+N_j}) \geq L_3(M^*) \left(\diam \beta_{S_n+N_j}\right)^{2-\alpha}
\; .\]
These estimates together yield the claim of the Lemma.
\end{proof}

We will use the symbol $\overset{M^*}{\sim}$ to join quantities which are
equivalent with positive multiplicative constants which depend on
$M^*$.  
\begin{lem}\label{lem:29xa,1}
  For any $n>1$,
  \[\nu(A_n) \overset{~M^*}{\sim} \left(\diam
  \beta_{S_n+N_3}\right)^{2-\alpha} \; .\]
\end{lem}
\begin{proof}
This follows straight from the definition of the conformal measure
given that $f^{S_n-1}$ maps $\beta_{S_n+N_3}$ onto $b_{N_3}$ with
distortion bounded in terms of $M^*$.
\end{proof}

\begin{lem}\label{lem:26xa,1}
  For any $n>1$,
  \[ \sum_{j} \left( \diam \zeta_{n,j} \right)^{2-\alpha} < K_2(M^*)
  \left(\diam \beta_{S_n+N_3}\right)^{2-\alpha} .\]
\end{lem}
\begin{proof}
  Summing up over $j$ and using Lemma~\ref{lem:24xp,1} together
  the estimate of $\nu(A_n)$ given by Lemma~\ref{lem:29xa,1}
 yields the claim.
\end{proof}

\begin{lem}\label{lem:29xa,2}
  For any $\epsilon>0$ and $M^*$ there is $n_0$ such that whenver
  $n\geq n_0$
  \[ \nu\left(b_{S_n+N_3}\setminus\overline{b}_{S_{n+1}+N_3}\right)
  \leq \left(\diam b_{S_{n+1}+N_3}\right)^{2-\alpha-\epsilon} \; .\]
\end{lem}
\begin{proof}
The absolute value of the derivative of $f^{-1}$ on $A_n$ is bounded
above by $L_1(M^*)
\left(\diam\beta_{S_{n+1}+N_3}\right)^{1-\frac{1}{d}}$. So,
\[  \nu\left(b_{S_n+N_3}\setminus\overline{b}_{S_{n+1}+N_3}\right)
\leq
L_2(M^*)\left(\diam\beta_{S_{n+1}+N_3}\right)^{(2-\alpha)(\frac{1}{d}-1)}\left(\diam
\beta_{S_n+N_3}\right)^{2-\alpha} \]
\[ \leq L_2(M^*)\left(\frac{\diam
  \beta_{S_n+N_3}}{\diam \beta_{S_{n+1}+N_3}}\right)^{2-\alpha}
\left(\diam b_{S_{n+1}+N_3}\right)^{{2-\alpha}} \; .\]
The first factor is bounded by $(\diam\beta_{S_n+N_3})^{-\epsilon}$ by
Lemma~\ref{lem:26xp,3}. The second was obtained from
$\diam b_{S_{n+1}+N_3} \overset{~M^*}{\sim}
\left(\diam\beta_{S_{n+1}+N_3}\right)^{1/d}$ by roundness.
This concludes the proof.
\end{proof}

\begin{lem}\label{lem:29xp,1}
For every $\epsilon>0$ and $M^*$ there is $n_0$ such that if $n\geq
n_0$, then
\[ \nu(b_{S_n+N_3}) \leq \left(\diam
b_{S_n+N_3}\right)^{2-\alpha-\epsilon} \; .\]
\end{lem}
\begin{proof}
  By Lemma~\ref{lem:29xa,2}, for any $\epsilon_1>0$ and $n\geq n_0(\epsilon_1)$ 
  \[ \nu(b_{S_n+N_3}) \leq \sum_{j=1}^{\infty}
  \left(\diam b_{S_{n+j}+N_3}\right)^{2-\alpha-\epsilon_1}\; . \]
  By Proposition~\ref{prop:25xp,1}, this can be further bounded from above,
  \[ \nu(b_{S_n+N_3}) \leq
  \sum_{j=1}^{\infty}
  \exp\left[(-S_{n+j}\frac{\lambda}{d}+o_{M^*}(S_n))(2-\alpha-\epsilon_1)\right] \]
  \[ \leq 
  \frac{\exp\left[(-S_{n+1}\frac{\lambda}{d}+o_{M^*}(S_n))(2-\alpha-\epsilon_1)\right]}{1-\exp(-\lambda/d)}\;, \]
  from estimating the sum of the geometric progression. By
  Lemma~\ref{lem:26xp,1}, we further obtain
  \[ \nu(b_{S_n+N_3}) \leq (1-e^{-\lambda/d})^{-1}\left(\diam
  b_{S_n+N_3}\right)^{2-\alpha-\epsilon_1}\exp(o_{M^*}(S_n)) \; .\]
  By choosing $\epsilon_1<\epsilon$ and  $n_0$ sufficiently
  large, we can absorb the constant and factor $\exp(o_{M^*}(S_n))$ into
  the form of the estimate of Lemma~\ref{lem:29xp,1}.
\end{proof}

\paragraph{Consequences for first entry maps.}
We will now use these results to obtain an estimate for the domains of
first entry maps. Recall how, by Proposition~\ref{prop:24xa,1}, on every
component of its domain the map $\phi_{S_n+N_3}$ has a univalent
extension onto $b_{S_{\sigma(n)+N_1}}$. By composing that with the
first entry map into $b_{S_{\sigma(n)}+N_3}$, one get a univalent
extension onto $b_{S_{\sigma(n)}+N_3}$ with a further continuation
onto $b_{S_{\sigma(n)}+N_1}$. 

Recall that $|\cdot|$ is used to denote $2$-dimensional Lebesgue measure of sets in $\C$.
\begin{prop}\label{prop:29xp,1}
Suppose that $X$ is a component of the domain of the first entry map
$f^r$ into
$b_{S_{\sigma(n)}+N_3}$ such that $f^r$ from $X$ 
continues univalently to map onto $b_{S_{\sigma(n)}+N_1}$. Let $X_n$
by the intersection of $X$ with the domain of 
$\phi_{S_n+N_1}$. Then for every  $n>1$,  
\[ \frac{|X_n|}{|X|} \leq \exp\left(-\frac{S_n}{2}\frac{\lambda\alpha}{2d}
+ o_{M^*}(S_n)\right) \; .\]
\end{prop}
\begin{proof}
  By Lemma~\ref{lem:26xp,1} and from the roundness of pieces,
  \begin{equation}\label{equ:29xp,3}
    \log \left(\diam b_{S_{\sigma(n)+N_3}}\right)^{-1} = S_{\sigma(n)}\frac{\lambda}{d}
    + o_{M^*}(S_{\sigma(n)})
   \end{equation} 
  while by Proposition~\ref{prop:25xp,1}, any component $\zeta$ of the first
  entry map $\phi_{S_n+N_1}$ satisfies
  \[ \log \left( \diam\zeta\right)^{-1} \geq
  S_n\frac{\lambda}{d}+o_{M^*}(S_n)\; .\]
  Hence
  \begin{equation}\label{equ:29xp,1} \log\frac{\diam b_{S_{\sigma(n)+N_3}}}{\diam\zeta} \geq
  (S_n-S_{\sigma(n)})\frac{\lambda}{d}+o_{M^*}(S_n) =  
    \frac{S_n}{2}\frac{\lambda}{d} + o_{M^*}(S_n)
    \end{equation}
  since from Definition~\ref{defi:28xa,1}, $S_{\sigma(n)} =
  \frac{S_N}{2}+o_{M^*}(S_n)$. 

  Now observe that $X$ is mapped onto $b_{S_{\sigma(n)}+N_3}$ with
  distortion bounded in terms of $M^*$, since the mapping extends
  univalently onto $b_{S_{\sigma(n)}+N_1}$.

  Thus, if $\zeta \subset X$, estimate~(\ref{equ:29xp,1}) yields
   \begin{equation}\label{equ:29xp,2} \log\frac{\diam X}{\diam\zeta} \geq
  \frac{S_n}{2}\frac{\lambda}{d} + o_{M^*}(S_n)\; .
    \end{equation}

   Using the same mapping of $X$ onto $b_{S_{\sigma(n)}+N_3}$ with
   bounded distortion, we conclude from Lemma~\ref{lem:29xp,1} that
   \[ \nu(X) \leq L_1(M^*) \left(\diam X\right)^{2-\alpha}\left(\diam
   b_{S_{\sigma(n)}+N_3}\right)^{-\epsilon_1}\]
   for every $\epsilon_1>0$ provided that $\sigma_n$ is large
   enough.
   
From Lemma~\ref{lem:24xp,1}
\begin{equation}\label{equ:29xp,4}
  \sum_{\zeta\subset X_n} \left(\diam\zeta\right)^{2-\alpha}
\end{equation}
\[   \leq
L_2(M^*)\nu(X) \leq L_1(M^*) L_2
(M^*) \left(\diam
X\right)^{2-\alpha} \left(\diam b_{S_{\sigma(n)}+N_3}\right)^{-\epsilon_1} \; .\]

Since $b_{S_{\sigma(n)}+N_3}$ was round and $X$ is its preimage with
bounded distortion,
\[ |X| \geq L_3(M^*) \left(\diam X\right)^2 \]
and $|\zeta| \leq L_4 \left(\diam\zeta\right)^2$,
from estimates~(\ref{equ:29xp,4}) and~(\ref{equ:29xp,1}) one obtains
\[ \frac{|X_n|}{|X|} \leq\frac{L_4}{L_3(M^*)}  \frac{ \sum_{\zeta\subset X_n}
  \left(\diam\zeta\right)^{2-\alpha}}{\left(\diam X\right)^{2-\alpha}} \frac{\sup\{
\left(\diam\zeta\right)^{\alpha} :\: \zeta\subset X_n\}}{\left(\diam
  X\right)^{\alpha}}\]
\[ \leq 
L_4(M^*)~ \left(\diam b_{S_{\sigma(n)}+N_3}\right)^{-\epsilon_1}\exp\left(-\alpha\frac{S_n}{2}\frac{\lambda}{d}+o_{M^*}(S_n)\right)
\; .\]

By Lemma~\ref{lem:26xp,1},
\[ \diam b_{S_{\sigma(n)+N_3}} =
\exp\left(-\sigma(n)\frac{\lambda}{d}+o_{M^*}(S_{\sigma(n)})\right) \]
which leads to
\[ \frac{|X_n|}{|X|} \leq 
L_4(M^*)~\exp\left(-\alpha\frac{S_n}{2}\frac{\lambda}{d} +
\epsilon_1S_{\sigma(n)}\frac{\lambda}{d}+o_{M^*}(S_n)\right) \;.\]

Since the constants and $\epsilon_1 S_{\sigma(n)}\frac{\lambda}{d}$
can be rolled into $o_{M^*}(S_n)$, Proposition~\ref{prop:29xp,1}
follows.
\end{proof}

\subsection{Deep point}
Recall that $\zeta_{n,j}$ denoted the connected components of the
domain of $\phi_{S_n+N_1}$. 
Define $Z_n = \bigcup_{j} \zeta_{n,j}$ and $Z=\bigcup_n Z_n$.

\begin{theo}\label{theo:24xp,1}
  For any $\epsilon>0$ and if $M^*$ is sufficiently large
  \[ \lim_{r\rightarrow 0^+} \frac{|Z\cap
    \D(c_0,r)|}{r^{2+\frac{\alpha}{d}-\epsilon}}=0 
   . \]
\end{theo}

\paragraph{Step I.}
We show that for any $M^*$ and $\epsilon>0$
there is $n_0$ such that if $n\geq n_0$ then
\[ \frac{|Z_n|}{(\diam \beta_{S_n+N_3})^2} \leq
  \left(\diam(\beta_{S_n+N_3})\right)^{\frac{\alpha}{d}-\epsilon} \; .\]

\[ |Z_n| \leq L \sum_{j} \left(\diam\zeta_{n,j}\right)^2 \leq \left[
  \sum_{j} \left(\diam\zeta_{n,j}\right)^{2-\alpha} \right] \cdot
\left[ \sup_j \{ \diam \zeta_{n,j} \} \right]^{\alpha} \; .\]
The first factor on the right-hand side can be bounded by
Lemma~\ref{lem:26xa,1} and the second by Lemma~\ref{lem:26xp,2}
leading to
\[ |Z_n| \leq \left( \diam b_{S_n+N_3} \right)^{2-\alpha}\left(
  \diam b_{S_n+N_3}\right)^{\alpha(1+\frac{1}{d}-\epsilon)} \]
provided $n$ is large enough. 
Finally,
\[ |Z_n| \leq \left( \diam b_{S_n+N_3}\right)^{2 + \frac{\alpha}{d} -
  \alpha\epsilon} \; .\] 
By dividing both sides by $(\diam \beta_{S-n+N_3})^2$ and taking
into account $\alpha\leq 1$, we get the claim of Step I.

\paragraph{Step II.}
We will prove that for every $\epsilon>0$ and if $M^*$ is large
enough, there is $r_0$ so that if $0<r<r_0$, then
\[ \frac{|Z\cap \D(c_0,r)|}{|\D(c_0,r)|} < r^{\frac{\alpha}{d}-\epsilon}
\; .\]

When $M^*$ is sufficiently large, any circle centered at $c_0$
intersects at most two of the annuli $A_n$.
So pick $r>0$ and choose the largest $n(r)$ for which $A_{n(r)}
\subset \D(c_0,r)$. There are at most two annuli $A_{n(r)-1}$ and
$A_{n(r)-2}$ which also intersect $\D(c_0,r)$. By Lemma~\ref{lem:26xp,3} for
any $\epsilon_1>0$ and if $n(r)$ is large enough depending on $\epsilon_1, M^*$, 
$\diam \beta_{S_{n(r)-2}+N_3} \leq r^{1-\epsilon_1}$. Inserting this
into the estimate of Step I stated for some $\epsilon_2>0$ we get that
for $n(r)$ large enough and for $j=1,2$

\begin{equation}\label{equ:26xp,3}
  \frac{|Z_{n(r)-j}|}{r^{2}} \leq
        r^{(1-\epsilon_1)(\frac{\alpha}{d}-\epsilon_2)+2\epsilon_1} \; .
\end{equation}

By the Step I, for every $n\geq n(r)$ and any $\epsilon_3>0$
\begin{equation}\label{equ:26xp,2} \frac{|Z_{n}|}{|\D(c_0,r)|} \leq
L~\frac{\left(\diam\beta_{S_n+N_3}\right)^2}{r^2}
\left(\diam\beta_{S_n+N_3}\right)^{\frac{\alpha}{d}-\epsilon_3}\;,
\end{equation}
where $L$ is a geometric constant.  By Lemma~\ref{lem:26xp,1},
\[ \frac{\left(\diam\beta_{S_n+N_3}\right)^2}{r^2} \leq
d^{S_{n(r)}-S_n} (\diam\beta_{S_n+N_3})^{-\epsilon_4} \]
for any $\epsilon_4>0$ provided $n(r)$ is large enough.
By summing up the estimates~(\ref{equ:26xp,2}) for $n\geq n(r)$, we get
\[ \frac{|\bigcup_{n\geq n(r)} Z_n|}{|\D(c_0,r)|} \leq \frac{Ld}{d-1}
r^{\frac{\alpha}{d}-\epsilon_3-\epsilon_4} \; .\]
By picking $\epsilon_3,\epsilon_4$ as well as $\epsilon_1,\epsilon_2$
in estimate~(\ref{equ:26xp,3}) suitably small for the desired
$\epsilon$
and making $r$ small enough to
produce $n(r)$ correspondingly large to absorb the constants, we get the claim
of Step II.  

Theorem~\ref{theo:24xp,1} follows directly from Step II. 
\subsection{Estimates on the ray.}
Let $\gamma$ denote the external ray of ${\cal J}_{c_0}$ which
lands at $c_0$. Use notations $Z$ and $Z_n$ from the previous
section and let $\chi_Z$, etc, by the indicator functions.
$\lambda_2$ is the $2$-dimensional  Lebesgue measure of the plane. 

\begin{theo}\label{theo:30xa,1}
There exist a bound $o(R)$, $\lim_{R\rightarrow 0^+} o(R)=0$, and $R_0>0$ such
that for every $z_0 \in {\gamma}$ and if $|z_0-c_0|<R_0$
\[ \int_{D(z_0,R)} \frac{\chi_Z(w)\, d\lambda_2(w)}{|z_0-w|^2} \leq
o(R) \; .\]
\end{theo}

Let us start with a lemma.
\begin{lem}\label{lem:30xa,1}
Consider a domain $X$ as in the statement of
Proposition~\ref{prop:29xp,1} for $n>1$ and such that it intersects
the domain of $\phi_{S_n+N_1}$. There is a constant
$0<K(M^*)$ so that for every such $X,n$ and $z_0\in{\gamma}$   
\[ \dist(z_0,\overline{X}) \geq K(M^*)~ \diam X \; .\]
\end{lem}
\begin{proof}
By hypothesis, $X$ is surrounded by an extension domain which is mapped univalently
onto $b_{S_{\sigma(n)}+N_1}$ and by Corollary~\ref{coro:24xa,1}, the
extension domain does not contain $z_0$. Hence, we have an annulus $A$
which is conformally equivalent to
$b_{S_{\sigma(n)}+N_1}\setminus\overline{b}_{S_{\sigma(n)}+N_3}$ which
  contains $X$
  in the bounded component of its complement leaving $z_0$ in the
  unbounded one. From our construction, $\mod A \geq \frac{2}{d} M^*$
  so the claim follows by Teichm\"{u}ller's estimates, see~\cite{lehvi}.
\end{proof}

We will now present the proof in a sequence of steps.
\paragraph{Step I.}
Recall that for $n>1$ the first entry mapping $\phi_{S_n+N_1}$ has a
univalent extension from every
component of its domain which maps onto $b_{S_{\sigma(n)+N_3}}$ and
whose domain is contained in $A_n$. Let us denote the union of the
domains of such extensions by $\tilde{Z}_n$. We have $Z_n \subset
\tilde{Z}_n \subset A_n$ for each $n$. 

\begin{lem}\label{lem:I}
There exist a positive sequence $\epsilon(n)$ with
$\lim_{n\rightarrow\infty}\epsilon(n)=0$ and $n_0$ such that for every
  $n\geq n_0$ 
  and $z_0\in{\gamma}$ there are $0<\rho_1<\rho_2$ so that
  \[ \tilde{Z}_n \subset \{ z :\: \rho_1 < |z-z_0| <
  \rho_2 \}\]
  with
  \[ \log\frac{\rho_2}{\rho_1} \leq
  \epsilon(n)\log\left(\diam\beta_{S_n+N_3}\right)^{-1} \; .\]
\end{lem}
The remaining part of Step I is devoted to the proof of Lemma~\ref{lem:I} which will be divided into several geometric cases.
\subparagraph{The case of $z_0$ far away.}
The first case is when $|c_0-z_0|\geq 2\diam \beta_{S_n+N_3}$.
The $\rho_2/\rho_1 \leq 3$, while in view of Lemma~\ref{lem:26xp,1}
\[ \log\left(\diam\beta_{S_n+N_3}\right)^{-1} \geq L_1 S_n \]
with positive $L_1$ provided $n$ is large enough. Hence, in this case
to satisfy the claim we just need $\epsilon(n) \geq \frac{\log 3}{L_1
  S_n}$.

So, from now on,  suppose $|c_0-z_0| < 2\diam\beta_{S_n+N_3}$.
Then we can put
\begin{equation}\label{equ:30xa,2}
  \rho_2 = 3\cdot\diam\beta_{S_n+N_3} \; .
\end{equation}
\subparagraph{$z_0$ not too deep.}
 In this case we assume additionally that $z_0$ is outside
 $\beta_{S_{n+2}+N_3}$. 

In order to estimate $\rho_1$, let us quote the following
\begin{fact}\label{fa:30xa,1}
For every $z_0\in{\gamma}$, $D(z_0,|z_0-c_0|^{1+o(|z_0-c_0|)}) \cap
{\cal J} = \emptyset$.
\end{fact}
\begin{proof}
  This is a statement of asymptotic Lipschitz accessibility,
  see~\cite{gsarkiv}. A much stronger claim is provided by Theorem~\ref{theo:flat}.
\end{proof}

Thus, for $\rho_1' = \left| z_0-c_0\right|^{1+o(2\,\diam\beta_{S_n+N_3})}$, the ball
$\D(z_0, \rho_1)$ misses ${\cal J}$. With our extra hypothesis
$|z_0-c_0|\geq \diam\beta_{S_{n+2}+N_3}$ and so
\begin{equation}\label{equ:30xa,1}
  \rho_1' \geq
  \left(\diam\beta_{S_{n+2}+N_3}\right)^{1+o(2\,\diam\beta_{S_n+N_3})}\; .
\end{equation}

  Now suppose $X$ is any component of $\tilde{Z}_n$. By definition, it
  intersects ${\cal J}$. So by Lemma~\ref{lem:30xa,1},
  $ \rho_1 \geq L_2(M^*)\rho'_1$ where $L_2(M^*):=\frac{K(M^*)}{1+K(M^*)}$
  with $K(M^*)$ from that Lemma.

  Taking into account estimates~(\ref{equ:30xa,2})
  and~(\ref{equ:30xa,1}), we arrive at
  \[ \frac{\rho_2}{\rho_1} \leq \frac{3}{L_2(M^*)}
  \frac{\diam\beta_{S_n+N_3}}{\diam\beta_{S_{n+2}+N_3}}
  \left(\diam\beta_{S_n+N_3}\right)^{o(2\,\diam\beta_{S_n+N_3})} \; .\]

  Taking logarithms and using Lemma~\ref{lem:26xp,1}, we get
  \[
  \frac{\log\frac{\rho_2}{\rho_1}}{\log\left(\diam\beta_{s_n+N_3})\right)^{-1}}\]
  \[ \leq
  \frac{\log\frac{3}{L_2(M^*)}+(S_{n+2}-S_n)\lambda+o_{M^*}(S_{n+2})+ S_n\lambda\,
  o_{M^*}(2\,\diam\beta_{S_n+N_3})}{S_n\lambda+o_{M^*}(S_n)}\; . \]

    We use that $\lim_{n\rightarrow \infty}\frac{S_{n+1}}{S_n}=1$. All terms in the numerator can be rolled into $o_{M^*}(S_n)$ and
    so the claim follows.
    \subparagraph{The case of $z_0\in\beta_{S_{n+2}+N_3}$.}
For $n$ sufficiently large $S_{n+2}+N_1 \geq S_{n+1}+N_3$ and so
$\beta_{S_{n+2}+N_3}$ is surrounded inside $\beta_{S_{n+1}+N_3}$ by an
annulus with modulus $M^*/d$. Hence,
\[ \rho_1 \geq L_3(M^*) \diam\beta_{S_{n+1}+N_3} \]
with $L_3(M^*)$ and for $\rho_2$ we can still take
estmate~(\ref{equ:30xa,2}). Hence,

\[ \frac{\log\frac{\rho_2}{\rho_1}}{\log (\diam\beta_{S_n+N_3})^{-1}} \leq
\frac{\log\frac{3}{L_3(M^*)} + (S_{n+1}-S_n)\lambda +
  o_{M^*}(S_{n+1})}{S_n\lambda + o_{M^*}(S_n)} \]
which tends to $0$ with $n$ as in the preceding case.

This completes the proof of Lemma~\ref{lem:I} and  Step I. 
\paragraph{Step II.}
Recall set $\tilde{Z}_n$ introduced in Step I. Let ${\cal X}_n$ be a
connected component of $\tilde{Z}_n$.
\begin{em}
Then for every $n>1$ and
$z_0\in{\gamma}$
\[ \frac{\int_{{\cal X}_n} \frac{\chi_{Z_n}(w)\,
    d\lambda_2(w)}{|w-z_0|^2}}
{\int_{{\cal X}_n} \frac{d\lambda_2(w)}{|w-z_0|^2}} \leq  \exp\left(-\frac{S_n}{2}\frac{\lambda\alpha}{2d}
  + o_{M^*}(S_n)\right) \; .\]
\end{em}

  By Lemma~\ref{lem:30xa,1},
  $\frac{|w_1-z_0|}{|w_2-z_0|} \geq \frac{K(M^*)}{1+K(M^*)}$ for any
  $w_1, w_2\in {\cal X}_n$.

  Then, if $X_n = Z_n \cap {\tilde {\cal X}}_n$,
  \[  \frac{\int_{{\cal X}_n} \frac{\chi_{Z_n}(w)\,
    d\lambda_2(w)}{|w-z_0|^2}}
{\int_{{\cal X}_n} \frac{d\lambda_2(w)}{|w-z_0|^2}} \leq
  \left(1+K(M^*)^{-1}\right)^2 \frac{|X_n|}{|X|} \]
  and the estimate then follows directly from
  Proposition~\ref{prop:29xp,1}, since the constant can be rolled into
  $o_{M^*}(S_n)$. 

  \paragraph{Step III.}
  \begin{em}
  For every $n>1$ and $z_0\in {\gamma}$

  \[ \int_{\CC} \frac{\chi_{Z_n}(w)\,
    d\lambda_2(w)}{|w-z_0|^2} \leq \epsilon(n) \exp\left(-\frac{S_n}{2}\frac{\lambda\alpha}{2d}
  + o_{M^*}(S_n)\right) \]
  where $\epsilon(n)$ is the sequence from Step I. 
  \end{em}

  Since every component of $Z_n$ is contained in some ${\cal X}_n$, Step II and  Proposition~\ref{prop:29xp,1} imply

  \[ \int_{\CC} \frac{\chi_{Z_n}(w)\,
    d\lambda_2(w)}{|w-z_0|^2} \leq \exp\left(-\frac{S_n}{2}\frac{\lambda\alpha}{2d}
  + o_{M^*}(S_n)\right) \int_{\tilde{Z}_n}
  \frac{d\lambda_2(w)}{|w-z_0|^2} \; .\]
  By Step I,
  \[  \int_{\tilde{Z}_n}
  \frac{d\lambda_2(w)}{|w-z_0|^2} \leq \int_{\rho_1<|u|<\rho_2}
  \frac{d\lambda_2(u)}{|u|^2} = \log\frac{\rho_2}{\rho_1} \leq
  \epsilon(n)\log\left(\diam  \beta_{S_n+N_3}\right)^{-1} \; .\]

  By Lemma~\ref{lem:26xp,1},
  \[ \log\left(\diam \beta_{S_n+N_3}\right)^{-1} = \lambda S_n +
  o_{M^*}(S_n) \; .\]

  Taking all these estimates together yields
  \[  \int_{\CC} \frac{\chi_{Z_n}(w)\,
    d\lambda_2(w)}{|w-z_0|^2} \leq \epsilon(n)\left(\lambda S_n + o_{M^*}(S_n)\right) \exp\left(-\frac{S_n}{2}\frac{\lambda\alpha}{2d}
  + o_{M^*}(S_n)\right) \]
  which gives the claim of Step III, since the factor before the
  $\exp$ involving $S_n$ can be included in the $o_{M^*}(S_n)$ in the
  exponent.

  \paragraph{Step IV.}
  \begin{em}
    For any $n>1$, $\dist(\tilde{Z}_n,{\gamma}) > 0$.
  \end{em}

  Clearly, $\dist({\cal J}\setminus\beta_{S_{n+1}+N_3},\gamma) := D>0$.
  By Lemma~\ref{lem:30xa,1} applied to each component of
  $\tilde{Z}_n$,
  \[ \dist\left(\tilde{Z}_n,{\gamma}\right) \geq \frac{K(M^*)D}{1+K(M^*)} \;
      .\]
  \paragraph{Conclusion of the proof of Theorem~\ref{theo:30xa,1}.}
  Choose $R_0 < \dist(c_0, A_1)$. Then the claim of Step IV also holds
  for $n=1$. We conclude that there is a function $n(R)$,
  $\lim_{R\rightarrow 0^+} n(R) = \infty$ such that for any $z_0 \in
  {\gamma} \cap \D(c_0,R_0)$,  the disk $\D(z_0,R)$ is disjoint from $Z_n$
  for all $n<n(R)$.

  Then by Step III,

  \[ \int_{D(z_0,R)} \frac{\chi_{Z}(w)\, d\lambda_2(w)}{|w-z_0|^2}
  \leq \sum_{n=n(R)}^{\infty} \exp\left(-\frac{S_n}{2}\frac{\lambda\alpha}{2d}
  + o_{M^*}(S_n)\right) := o(R) \; .\]

  Since the series
  \[ \sum_{n=1}^{\infty}  \exp\left(-\frac{S_n}{2}\frac{\lambda\alpha}{2d}
  + o_{M^*}(S_n)\right) < \infty\; , \]
  the bound $o(R)$ tends to $0$ with $R\rightarrow 0$.

  \section{Distortion estimates}
  \paragraph{The transversality function.}
  The transversality function is defined by
  \[ \mathcal{T}(c) = \sum_{n=0}^{\infty} \left(Df^n_c(c)\right)^{-1} \]
  wherever the series is convergent, which is at least for $c\notin {\cal M}_d$. 
  
  Fact~\ref{fa:1kp,1} follows from calculus and the definition of $\Psi$ and $\Psi_c$.
  \begin{fact}\label{fa:1kp,1}
    For $c\notin{\cal M}_d$ 
    \[ \frac{D_c\Psi^{-1}(c)}{D_z\Psi^{-1}_c(z)_{|z=c}  } = \mathcal{T}(c) .\]
  \end{fact}
  \paragraph{The main estimate.}
  Now let $u_0$ be the point on $\partial \D$ with the external argument of $c_0$.
  $u_n$ is chosen with the same argument as $u_0$ so that $\Psi(u_n)=c_n$
  is on the boundary of $V_{S_n+N_4}$. Hence,
  \begin{equation}\label{equ:1kp,1}
  |u_n-u_0| \leq K(M^*) d^{-S_n-N_4} .
  \end{equation}
  Choose $u$ on the segment 
  between $u_0$ and $u_n$ and write $c(u) := \Psi(u), z(u)=\Psi_{c(u)}(u)$ and
  $z_n(u)=\Psi_{c(u)}(u_n)$. Observe also that $|u_n|^{S_n+N_4}$ is a
  fixed number corresponding to the equipotential which bounds the
  initial Yoccoz piece. 

  Our goal is proving the following.
  \begin{prop}\label{prop:3kp,1}
  For almost every $c_0\in\partial{\cal M}_d$ with respect to the
  harmonic measure   
  \[ \lim_{n\rightarrow\infty} 
  \log\frac{D_z\Psi_{c_n}(z)_{|z=u_n}}{D_z\Psi_{c_0}(z)_{|z=u_n}} = 0 .\]
  \end{prop}

  The idea of the proof is to estimate the derivative of $D_u \log
  D_z\Psi_{c(u)}(z)_{z=u_n}$ for $u$ between $u_0$ and $u_n$. It is more
  convenient to take the derivatives with respect to $c$ instead at
  $c=c(u)$, so begin by estimating $D_u c(u) = D\Psi(u)$.

  \paragraph{Estimate of $D_u\Psi(u)$.}
  \begin{lem}\label{lem:4ka,1}
  For almost every $c_0\in\partial{\cal M}_d$ in the sense of the
  harmonic measure that exists $Q(c_0)$ so that for all $n\geq
  n_0(c_0)$ and $M^* \geq M^*_0$
  \[ |u_n - u_0| |D_u\Psi(u)| \leq \frac{\left| Q(c_0)\right|}{ \left| D_z
      f_{c(u)}^{S_n}\left(z_n(u)\right)\right|} .\]
  \end{lem}
  \begin{proof}
  Calculate
  \[ D_u\Psi(u) = \left( D_c\Psi^{-1}(c)_{|c=c(u)}\right)^{-1} =
  \frac{D_z\Psi_{c(u)}(z)_{|z=u}}{\mathcal{T}\left(c(u)\right)}    \]
  from Fact~\ref{fa:1kp,1}.

  Then,
  \[ D_z\Psi_{c(u)}(z)_{|z=u} = d^{S_n} u^{d^{S_n}-1} \cdot D_z
  \Psi_{c(u)}(z)_{|z=u^{d^{S_n}}} \cdot \left( D_z
  f^{S_n}_{c(u)}(z)_{|z=c(u)} \right)^{-1} .\]

  By Koebe's one-quarter lemma, for $n$ large enough, 
  \begin{equation}\label{equ:4kp,1} \left |  D_z \Psi_{c(u)}(z)_{|z=u^{d^{S_n}}} \right|  \leq K_1
  d^{N_4} \dist\left({\cal  J}_{c(u)},
  f^{S_n}{c(u)}\right) \leq K_2 d^{N_4} \end{equation}
  since $f^{S_n}\left(c(u)\right)$ belongs to a bounded set fixed
  by the Yoccoz puzzle construction. 

  From estimates~(\ref{equ:1kp,1},\ref{equ:4kp,1})
  and the fact that $\left| \log{\T} \right|$ is bounded on
  almost every external ray of ${\cal M}_d$, see Theorem~1.2 \cite{grsw} or apply the Abel theorem,
  \[ |u_n - u_0| |D_u\Psi(u)| \leq \frac{Q_1(c_0)}{\left| D_z
    f^{S_n}_{c(u)}(z)_{|z=c(u)} \right|} .\]

  Finally, the point $z=c(u)$ can be replaced by $z=z_n(u)$ by the
  bounded distortion of $f^{S_n}_{c(u)}$ on $\beta_{S_n+N_4,c(u)}$
  which holds if $M^*$ is sufficiently large.
    
  \end{proof}
      
    \paragraph{Proof of the Proposition.}
Start with $\Psi_c(z) = f^{-S_n}_c \circ \Psi_c \left(z^{d^{S_n}}\right)$ which
leads to 
    \[ \log D_z\Psi_{c}(z)_{z=u_n} = \log D_z f^{-S_n}_c\circ
    \Psi_c\left(u_n^{d^{S_n}}\right) + \log D_z
    \Psi_c\left(u_n^{d^{S_n}}\right) + \log d^{S_n} u_n^{d^{s_n}-1} .\] 

    Under differentiation with respect to $c$ the last term drops
    out. The derivative of the second one is bounded independently of
    $n$, thus after multiplying by $D\Psi(u)$ and integrating from
    $u_0$ to $u_n$ it goes to $0$ with $n$ by
    Lemma~\ref{lem:4ka,1}. Hence, only the first term requires
    closer attention.  Recall the nonlinearity $\mathfrak{n}f=\frac{D^2f}{Df}$.

    \begin{multline*} D_c \log D_z f_c^{-S_n}\circ\Psi_c \left(u_n^{d^{S_n}}\right) =
    (\mathfrak{n}f^{-S_n}_c)\left(\Psi_c
    \left(u_n^{d^{S_n}}\right)\right)
    D_c\Psi_c\left(u_n^{d^{S_n}}\right) -\\  D_c\log D_z
    f^{S_n}_c(z)_{|z=z_n(u)} .\end{multline*}
    Here the first term is bounded independently of $n$ by the bounded
    distortion of $f_c^{S_n}$ and goes to $0$ after integrating from
    $u_0$ to $u_n$ so we concentrate on the last one. 

    \[  D_c\log D_z   f^{S_n}_c(z)_{|z=z_n(u)} = \sum_{k=1}^{S_n-1}
    (\mathfrak{n}f_n^{S_n-k})\left((f^k_c\left(z_n(u)\right)\right) .\]

    Since the distortion of $f^{k-S_n}_c$ is bounded and in fact can be
    made as small as needed by adjusting construction parameters
    $M^*$, for every $\varepsilon>0$ and $n$ large enough

    \[ \left| D_c\log D_z   f^{S_n}_c(z)_{|z=z_n(u)} \right| \leq
    \varepsilon \sum_{k=1}^{S_n-1}
    \left| D_zf^{S_n-k}\left(f^k_c\left(z_n(u)\right)\right) \right| .\]

    By Lemma~\ref{lem:4ka,1}, the integral of this term when $u$
    changes from $u_0$ to $u_n$ is bounded by
    \[ Q(c_0)\varepsilon \sum_{k=1}^{S_n-1} \left| D_z
    f^k_{c(u)}(z)_{|z=z_n(u)} \right|^{-1} .\]
    Since the sum in this formula is uniformly bounded by  Theorem~1.2 of~\cite{grsw} 
    and
    $\varepsilon$ can be arbitrarily small,
    Proposition~\ref{prop:3kp,1} is proved.  

\paragraph{Proof of Theorem~\ref{theo:trans1}.}
We write using one after another Theorem~\ref{theo:continuity}, then
Proposition~\ref{prop:3kp,1} and Fact~\ref{fa:1kp,1}:

\begin{multline}\label{equ:17ja,1} \frac{1}{D_z\Upsilon_{c_0}(z)_{|z=c_0}} = \lim_{n\rightarrow\infty}
\frac{1}{D_z\Upsilon_{c_0}(z)_{|z=\Psi_{c_0}(u_n)}} =
\lim_{n\rightarrow\infty} \frac{D_z\Psi_{c_0}(z)_{z=u_n}}{D\Psi(u_n)} 
= \\ \lim_{n\rightarrow\infty} \left[
  \frac{D_z\Psi_{c_0}(z)_{z=u_n}}{D_z\Psi_{c_n}(z)_{z=u_n}}
  \frac{D_z\Psi_{c_n}(z)_{z=u_n}}{D\Psi(u_n)} \right]  =
\lim_{n\rightarrow\infty}  \frac{D_z\Psi_{c_n}(z)_{z=u_n}}{D\Psi(u_n)}
=\\ \lim_{n\rightarrow\infty} \frac{D\Psi^{-1}(c_n)}{D_z
  \Psi^{-1}_{c_n}(z)_{z=c_n}} = \lim_{n\rightarrow\infty} \T(c_n) .
\end{multline}

Let us recall now Theorem 1.2 of~\cite{grsw} which says that for
Lebesgue almost every $\alpha\in[0,1]$ there exist $K>0$ and
$\lambda>1$ such that for every $n>0$ and $c\in \Gamma(\alpha)$ where $\Gamma(\alpha):=\overline{\Psi \left\{re^{2\pi
  i\alpha} :\: r>1\right\}}$ the estimate $|D_z f_c^n(c)| \geq K\lambda^n$
holds.

Hence, for any such $\alpha$ and $c_0 = \lim_{r\rightarrow 1^+}
\Psi\left( re^{2\pi i\alpha} \right)$
\[ \lim_{c\rightarrow c_0, c\in\Gamma(\alpha)} \T(c) = \T(c_0) \]
by the Lebesgue dominated convergence theorem applied the sum in
formula~(\ref{equ:26mp,1}).   

Taking into account equation~(\ref{equ:17ja,1}) the proof of Theorem~\ref{theo:trans1} is complete.

  \section{Appendix}   

\subsection{Proof of claim {\rm (iv)} from Fact~\ref{theo:basic}.}
By Proposition~4 of~\cite{fine}, every point $z$ from $\partial Z\setminus {\cal J}_{c_0}$ is contained in Yoccoz piece $Y_z$ so that
$z$ is separated from the boundary of  $Y_z$ by an annulus of the modulus $m$. By the construction of box domains in~\cite{fine}, the boundary of $Y_z$ consists of a finite number of  pieces of the fixed Green equipotential line and hyperbolic geodesics of $\C\setminus {\cal J}_{c_0}$. If $\gamma$ from $\rm{(iv)}$ of  Fact~\ref{theo:basic} intersected $Y_z$ then, from the first part of \rm{(iv)}, it would have to land at a point of ${\cal J}_{c_0}$ which does not belong to the interior of $Z$. This would mean that the landing point $c_0$ of  $\gamma$ is  non-recurrent point, a contradiction. Therefore, $\gamma$ is disjoint from $Y_z$ and by  Teichmuller's module theorem, for every $\xi \in \gamma$ close enough to $c_0$,
$$d_{H}(\xi, {\cal J}_{c_0})\leq (1+C e^{-m})\; d_{H}(\xi,Z), $$
$C>0$ is a universal constant if $m>5$. Since $m\rightarrow \infty$ when $\xi\in \gamma$ tends to $c_0$, the limit from 
$\rm{(iv)}$ must be $1$.

\subsection{Proof of Theorem~\ref{theo:hedgehog}}
Non-hyperbolic systems are often studied by taking piecewise defined iterates of the map  which have some expansion and bounded distortion properties. In the case of uni-critical polynomials this leads to the construction of induced sequences of {\em box mappings},~\cite{hipek}. We will follow the description of induced box dynamics for generic parameters $c\in \partial \M$ with respect to the harmonic measure obtained in~\cite{fine}. The picture is largely simplified due to the fact that almost all returns are {\em non-close} and that only dynamics of the central branches is needed to prove the existence of hedgehogs. Fact~\ref{fact:28ha,1} follows directly from the definition of the induced box mappings and Proposition~7 from~\cite{fine}.


\begin{fact}\label{fact:28ha,1}
If  $c^* \in {\cal M}_d$ is typical with respect to the harmonic mesure then there is an
infinite induced sequence of proper analytic maps
$(\psi_{p,c^*})_{p=0}^{\infty}$ of degree $d$ with only one critical point at $0$, with their ranges $B_{p,c^*}$
and domains $B_{p+1,c^*}$ which are Jordan disks for all $p$ and satisfy $\overline{B_{p+1,c^*}} \subset
B_{p,c^*}$. Every $\psi_{p,c^*}$ is an iterate of $f_{c^*}$. 

Moreover,
\begin{itemize}
\item the sequence of $(\psi_{p,c^*})_{p=0}^{\infty}$ shows an exponential  decay of geometry,
\[ m_p(c^*):= \mod\hspace{-0.01cm}\left( B_{p,c^*} \setminus
\overline{B_{p+1,c^*}}\right) \leq \lambda_{c^*}^p\; ,\]
where $\lambda_{c^*}>1$,
\item every $\psi_{p,c^*}$ has a proper analytic extension of the degree $d$  to $ D_{p,c^*}$, $B_{p,c^*}\supset D_{p,c^*}\supset B_{p+1,c^*}$,  with the range $B_{p-1,c^*}$.
\end{itemize}

\end{fact} 
We are ready to prove Theorem~\ref{theo:hedgehog}. Let $c^*\in \partial \M$ be a typical parameter with respect to the harmonic measure and 
$(\psi_{p,c^*})_{p=0}^{\infty}$ the corresponding induced sequence. 
Since $c^*$ is fixed, we will drop it from the notation whenever there is no confusion.

Let $\epsilon$ and $m$ be the parameters from the definition of hedgehog neighborhoods, Theorem~\ref{theo:hedgehog}.  We choose  a large $k<p$ so that $1/d^k<\epsilon/10$ and $m_p(c^*)\geq 10m d^{k}$. Since ${\cal J}_{c^*}$ is connected, there is a continuum ${\cal C}_p\subset {\cal J}_{c^*}$ which traverses 
$B_{p-1} \setminus \overline{B_{p}}$ for $p$ large enough. Let us put 
$$\Phi_{k,p}=\psi_{p}\circ\dots\circ\psi_{p+k-1}$$ which is a proper analytic of 
map of degree $d^k$ on the annulus  $A=\Phi_{k,p}^{-1}(B_{p-1} \setminus \overline{B_{p}})\subset D_{p+k}$. Therefore, the modulus of $A$ is at least $10m$.

Since for every $p>0$, $\psi_{p}(0)\in B_{p-1}$ and  $\psi_{p}$ is  the composition of $z^d$ with 
a univalent map of a  vanishing distortion when $p$ tends to $\infty$, the preimages  $\Phi_{p,k}^{-1}({\cal C}_p)$ form a $5/d^k$-net relative to the size of the annulus  
$A$.
We have constructed a $(m,\epsilon)$-hedgehog layer around
$c^*\in {\cal J}_{c^*}$. Since hedgehog neighborhoods are quasiconformal invariants, their  existence at $c^*\in \partial \M$
follows from Fact~\ref{theo:basic}.

\end{document}